\documentclass[12pt]{amsart}

\usepackage{amsmath, amsthm}
\usepackage{graphicx,epsfig}
\usepackage{amsmath,latexsym,amssymb,verbatim}

\newtheorem{definition}{Definition}

\newtheorem{proposition}{Proposition}
\newtheorem{theorem}{Theorem}

\newtheorem*{theoreme}{Theorem 4}
\newtheorem{remark}{Remark}
\newtheorem{example}{Example}

\def\sideremark#1{\ifvmode\leavevmode\fi\vadjust{\vbox to0pt{\vss 
      \hbox to 0pt{\hskip\hsize\hskip1em           
 \vbox{\hsize2cm\tiny\raggedright\pretolerance10000
 \noindent #1\hfill}\hss}\vbox to8pt{\vfil}\vss}}}%

\renewcommand\thefootnote{\arabic{footnote}}
\begin{document}

\title[$\varepsilon$-neighborhoods of orbits, cohomological equations]{$\varepsilon$-neighborhoods of orbits of\\ parabolic diffeomorphisms and\\ cohomological equations}
\address{University of Zagreb, Faculty of Electrical Engineering and Computing, Department of Applied Mathematics, Unska 3, 10000 Zagreb, Croatia}
\email{maja.resman@fer.hr}

\address{Universit\' e de Bourgogne, Institut de Math\' ematiques de Bourgogne, UMR 5584 du CNRS, 9 avenue Alain Savary, 21078 Dijon Cedex, France} 

\author{Maja Resman}
\let\thefootnote\relax\footnotetext{The research was done at the Institute of Mathematics of the University of Burgundy, Dijon, and funded by the French government scholarship for the academic year 2012/13 and the 2012 Elsevier/AFFDU grant.}

\begin{abstract}
In this article, we study analyticity properties of (directed) areas of $\varepsilon$-neighborhoods of orbits of parabolic germs. The article is motivated by the question of analytic classification using $\varepsilon$-neighborhoods of orbits in the simplest formal class.

We show that the coefficient in front of $\varepsilon^2$ term  in the asymptotic expansion in $\varepsilon$, which we call the principal part of the area, is a sectorially analytic function in the initial point of the orbit. It satisfies a cohomological equation similar to the standard trivialization equation for parabolic diffeomorphisms. 

We give necessary and sufficient conditions on a diffeomorphism $f$ for the existence of globally analytic solution of this equation. Furthermore, we introduce new classification type for diffeomorphisms implied by this new equation and investigate the relative position of its classes with respect to the analytic classes. 
\end{abstract}

\maketitle


\bigskip 

\section{Introduction and main results}\label{one}

\subsection{Motivation}
Each germ of a parabolic diffeomorphism in the complex plane, 
\begin{equation}\label{diffe}
f(z)=z+a_1z^{k+1}+a_2z^{k+2}+o(z^{k+2}),\ k\in\mathbb{N},\ a_i\in\mathbb{C},\ a_1\neq 0,
\end{equation} can, by formal changes of variables, be reduced to the formal normal form, which is the time-one map of the holomorphic vector field
\begin{equation*}
f_0(z)=\exp(X_{k,\rho}),\ X_{k,\rho}=\frac{z^{k+1}}{1+\frac{\rho}{2\pi i}z^k}\frac{d}{dz},
\end{equation*}
for an appropriate choice of $k\in\mathbb{N}$ and $\rho\in\mathbb{C}$. The formal class of a parabolic diffeomorphism is given by the pair $(k,\rho),\ k\in\mathbb{N},\ \rho\in\mathbb{C}$. Here, $k$ is the same as in \eqref{diffe} and $\rho$ can, in the course of formal changes of variables reducing $f$ to $f_0$, be expressed using $k$ and the first $k+1$ coefficients $a_1,\ldots,a_{k+1}$. We have shown in \cite{resman} that the formal class of a diffeomorphism can be recognized \emph{looking} only at the (directed) area of the $\varepsilon$-neighborhood of one of its orbits. Accordingly, only finitely many terms in its asymptotic expansion in $\varepsilon$ determine the formal class of the diffeomorphism. 

On the other hand, the analytic class of a parabolic diffeomorphism is given by $2k$ diffeomorphisms, the so-called \emph{\' Ecalle-Voronin moduli} or \emph{horn maps}, see e.g. \cite{ecalle}, \cite{voronin} or \cite{loray}. This article was motivated by the following question: 

\smallskip
\emph{Can we read the analytic class of a diffeomorphism from the $\varepsilon$-neighborhoods of its orbits, regarded as functions of parameter $\varepsilon>0$ and of the initial point $z\in\mathbb{C}$?}
\smallskip 

It is clear that the analytic class, unlike the formal class, cannot be read from any finite jet of parabolic germ, see e.g.\cite[21]{ilya}. Therefore, we are forced to analyse the whole functions of $\varepsilon$-neighborhoods of orbits, not just finitely many terms in the expansion. The article does not answer the above question, but it gives different partial results summarized in Subsection~\ref{onetwo}, concerning the analyticity properties of the (directed) area of the $\varepsilon$-neighborhoods of orbits. We reach the conclusion that the principal part in the expansion, defined in Subsection~\ref{onetwo} below, is the only sectorially analytic object in the expansion. Moreover, it satisfies a cohomological equation similar to the trivialisation (Abel) equation, standardly used in context of analytic classification of germs. Except standard Abel equation, other cohomological equations and their relation to conjugacy problems have been considered for real-line diffeomorphisms in works of Belitskii and Tkachenko, Lyubich, Grintchy and Voronin, \cite{beli}, \cite{lyu}, \cite{vorgri}. In this article, the cohomological equation appears in a different manner, from geometric properties of $\varepsilon$-neighborhoods of orbits. Difference equations and the question of sectorial summability of their formal solutions (the so-called Stokes phenomenon) appear frequently in problems in nature, see e.g. \cite{jpr} for some insight. 

We study solutions of so-called $m$-Abel equations of the form $H\circ f-H=g$, $g(z)=z^m$, and their relation to analytic classificaton problem in Section~\ref{seven}. We show that the sectorial solutions of $1$-Abel equation for a germ are not sufficient to read its analytic class. We introduce new classification of diffeomorphisms with respect to $1$-Abel equation and show \emph{transversality} of these classes to the analytic classes. The question that is posed for the future is if the same property holds for higher cohomological equations.

Our main tool, the asymptotic behavior of $\varepsilon$-neighborhoods of sets (also called \emph{tube functions} in literature), was exploited in series of problems so far. The first term in the asymptotic expansion is related to the notion of box dimension and Minkowski content, see e.g. \cite{tricot} for exact definitions. Computed for appropriate invariant sets, they show intrinsic properties of dynamical systems. In the famous Weyl-Berry conjecture, the box dimension and the Minkowski content of the boundary for Laplace equation are related to the eigenvalue counting function, see \cite{lappom}. In discrete systems, box dimension and Minkowski content of orbits which accumulate at a fixed point reveal multiplicity of the generating function, moment of bifurcation or the complexity of bifurcation, see \cite{lana}, \cite{cheby}, \cite{belgproc}.   

In previous article \cite{resman}, we have proven that more terms in the asymptotic expansion of $\varepsilon$-neighborhoods are needed to read the formal class of a complex germ. The natural continuation was to investigate if analytic class can be seen in functions of $\varepsilon$-neighborhoods of orbits, in order to see to what extent $\varepsilon$-neighborhoods of orbits describe the germ.
\smallskip

\subsection{Definitions and main results}\label{onetwo}\

Let $$f(z)=\lambda z+a_1 z^{k+1}+a_2z^{k+2}+o(z^{k+2}),\ a_i\in\mathbb{C},\ k\in\mathbb{N},$$ $\lambda=\exp(2\pi i m/n),\ m,\ n\in\mathbb{N},$ be a parabolic diffeomorphism. Without loss of generality, in the article we assume that $\lambda=1$. Otherwise, instead of $f$, we consider its appropriate iterate, $f^{(\circ n)}$. Near the origin, the orbits of $f$ form the so-called \emph{Leau-Fatou flower}, see e.g. \cite{loray} or \cite{milnor}. In short, there exist $k$ attracting and $k$ repelling petals, around equidistant repelling and attracting directions. Petals are domains accumulating on $0$, bisected by attracting(repelling) direction and tangent to two repelling(attracting) directions at the origin. Attracting and repelling directions are normalized complex numbers $(-a_1)^{-1/k}$,\ $a_1^{-1/k}$ respectively. Orbits are tangent to attracting or repelling directions at the origin, see Figure~\ref{petal}. 

\begin{figure}[ht]\label{petal}
\hspace{1.2cm}\includegraphics[scale=0.3]{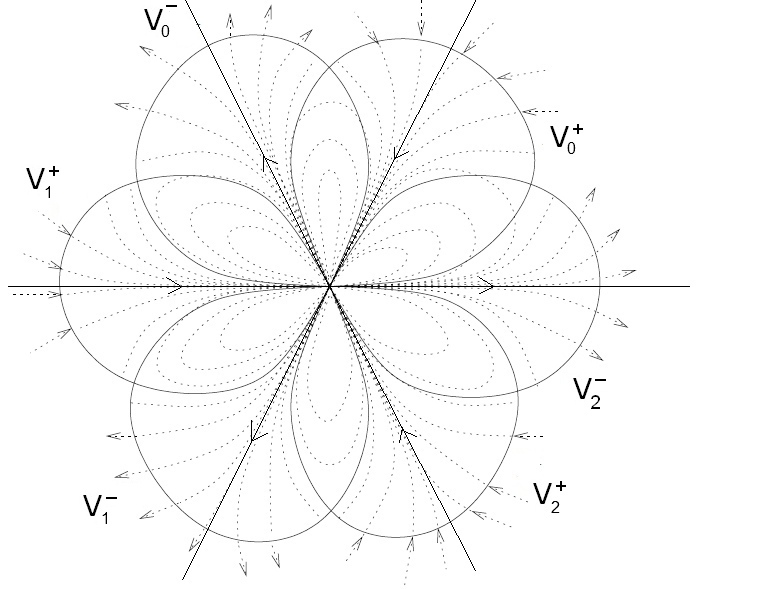}
\caption{\scriptsize{Attracting and repelling petals and directions for e.g. $f(z)=z+z^4+o(z^4)$}.}
\end{figure}

Let $V_+$ denote any attracting petal of $f(z)$. Let $$S^f(z)=\{z_n\ |\ z_n=f^{\circ n}(z),\ n\in\mathbb{N}\}$$ denote the orbit of $f$ with the initial point $z$ lying in $V_+$. Using difference equation $z_{n+1}-z_n=a_1 z^{k+1}+a_2 z^{k+2}+o(z^{k+2})$, the standard expansion for $f^{\circ n}(z)$ follows, see e.g. \cite{milnor}: 
$$
f^{\circ n}(z)=(-ka_1)^{-\frac{1}{k}}\cdot n^{-\frac{1}{k}}+o(n^{-\frac{1}{k}}),\ n\to\infty.
$$
Further expansion can be found in \cite{resman}.

\begin{definition}\label{dirarea}$[$see \cite{resman}$]$ Let $S^f(z),\ z\in V_+,$ be an attracting orbit of $f(z)$ with initial point $z$. Let $S^f(z)_\varepsilon$ denote its $\varepsilon$-neighborhood. \emph{Directed area of the $\varepsilon$-neighborhood of the orbit $S^f(z)$} is the complex number
$$
A^\mathbb{C}(z,\varepsilon)=A(S^f(z)_\varepsilon)\cdot t_{S^f(z)_\varepsilon},
$$
where $A(S^f(z)_\varepsilon)$ denotes the area and $t_{S^f(z)_\varepsilon}$ the center of the mass of the $\varepsilon$-neighborhood of the orbit.
\end{definition}
Here, for the sake of convenience, the directed area is defined in a slightly different manner than in \cite{resman}. In \cite{resman}, the center of mass was replaced by the normalized center of mass, $\frac{t_{S^f(z)_\varepsilon}}{|t_{S^f(z)_\varepsilon}|}$.
\medskip

 Let us remind the asymptotic expansion of $A^\mathbb{C}(z,\varepsilon),\ z\in V_+,$ from \cite{resman}, as $\varepsilon\to 0$:
\begin{align}\label{asy}
A^\mathbb{C}(z,\varepsilon)=&q_1\varepsilon^{1+\frac{2}{k+1}}+q_2\varepsilon^{1+\frac{3}{k+1}}+\ldots+q_{k-1}\varepsilon^{1+\frac{k}{k+1}}+q_k\varepsilon^{2}\log\varepsilon+\nonumber \\+H^{f,V_+}(z)\varepsilon^2+
&q_{k+1}\varepsilon^{2+\frac{1}{k+1}}\log\varepsilon+R(z,\varepsilon),\ R(z,\varepsilon)=O(\varepsilon^{2+\frac{1}{k+1}}),\nonumber\\
&\hspace{2.5cm} k\in\mathbb{N},\ k>1,\ q_i\in\mathbb{C},\ i=1,\ldots k+1.
\end{align}
Due to the modification in definition with respect to \cite{resman}, the exponents are shifted by $\frac{1}{k+1}$, but the proof is essentially the same. Let us remark here that the above expansion and formulas for the coefficients given in \cite{resman} hold in the case $k>1$. In the special case when $k=1$, we have the expansion:
\begin{align*}
A^\mathbb{C}(z,\varepsilon)=q_1&\varepsilon^2\log\varepsilon+H^{f,V_+}(z)\varepsilon^2+\\
&+q_2\varepsilon^{\frac{5}{2}}\log\varepsilon+R(z,\varepsilon),\ R(z,\varepsilon)=O(\varepsilon^\frac{5}{2}),\ q_1,\ q_2\in\mathbb{C}.\nonumber
\end{align*}
The coefficients are in this case given by slightly different formulas than stated in \cite{resman}, but the properties of the expansion are the same. 
In above expansions, $q_1,q_2,\ldots,q_{k+1}$ are complex functions of coefficients of $f$ and \emph{do not depend on the initial point}. The coefficient $H^{f,V_+}(z)$ is the first coefficient that depends on the initial point $z$. It is a well-defined function in $z$ on $V_+$. 
\begin{definition}
\emph{The principal initial point dependent part of the directed area of the $\varepsilon$-neighborhoods of orbits in $V_+$} is the first coefficient $H^{f,V_+}(z)$ in the expansion~\eqref{asy} depending on the initial point $z$, regarded as a function of $z\in V_+$, $z\mapsto H^{f,V_+}(z)$.
\end{definition}
By abuse, for the sake of simplicity, we will call function $z\mapsto H^{f,V_+}(z)$ only \emph{the principal part of area for $f$ on $V_+$}. Naturally, on a repelling sector $V_-$, we define \emph{the principal part of area for $f^{\circ -1}$ on $V_-$} as the first coefficient that depends on the initial point in the expansion \eqref{asy} for the orbit $S^{f^{\circ -1}}(z),\ z\in V_-,$ of the inverse diffeomorphism $f^{\circ -1}$. It is regarded as function of $z\in V_-$. We denote it by $z\mapsto H^{f^{\circ -1},V_-}(z)$.
\medskip

Let us comment shortly on properties of $A^\mathbb{C}(z,\varepsilon)$, as function of $\varepsilon>0$ and $z\in V_+$. They justify why we concentrate on the principal part in $A^\mathbb{C}(z,\varepsilon)$, as the only part that displays analytic property. All the results are elaborated in Section~\ref{three}. We show in Proposition~\ref{nonasy} that, for fixed initial point $z\in V_+$, the remainder $R(\varepsilon,z)$ in \eqref{asy} cannot be expanded any further in power-logarithmic scale with respect to $\varepsilon$. Moreover, it has accumulation of singularities at $\varepsilon=0$, see Proposition~\ref{accu}. Furthermore, for $\varepsilon$ fixed, $A^\mathbb{C}(z,\varepsilon)$ is not a sectorially analytic function of $z\in V_+$, see Proposition~\ref{no}. 

On the other hand, we prove in Section~\ref{three} the following Theorem~\ref{ppart} about sectorial analyticity of the principal parts of area. For simplicity, we consider only the germs from the model formal class $(k=1,\lambda=0)$, that is, formally equivalent to the model diffeomorphism $f_0$, $f_0(z)=Exp(z^2\frac{d}{dz})=\frac{z}{1-z}$. Furthermore, we assume that $f$ is \emph{prenormalized}. That is, the first normalizing change of variables is already made, and further we admit only changes of variables tangent to the identity. Therefore, all such diffeomorphisms are of the form:
$$
f(z)=z+z^2+z^3+o(z^3).
$$
In this case, locally there exists only one attracting petal $V_+$, invariant for $f$ (around negative real axis) and one repelling petal $V_-$, invariant for $f^{\circ -1}$ (around positive real axis). We denote the functions $H^{f,V_+}$ and $H^{f^{\circ -1},V_-}$ simply by $H^f$ and $H^{f^{\circ -1}}$.
\medskip

\begin{theorem}[Properties of principal parts of areas for $f$]\label{ppart} The principal parts of areas, $H^f$ for $f$ and $H^{f^{\circ -1}}$ for $f^{\circ -1}$, are analytic functions on the attracting sector $V_+$ and on the repelling sector $V_-$ respectively. Moreover, $H^f$ and $H^{f^{\circ -1}}$ are, up to explicit constants independent of $f$, related to the unique sectorially analytic solutions without constant term, $H_+$ on $V_+$ and $H_-$ on $V_-$, of difference equation:
\begin{equation}\label{ha}
H(z)-H(f(z))=\pi z,
\end{equation}
The following explicit formulas hold:
\begin{align*}
&H_+(z)-\frac{\pi}{4}+i\pi^2=H^f(z),\quad z\in V_+,\\
&H_-(z)-\frac{\pi}{4}=\pi z-H^{f^{\circ -1}}(z),\quad z\in V_-.\nonumber
\end{align*}
\end{theorem}

The equation \eqref{ha} resembles to the trivialization equation 
\begin{equation}\label{abel}
\Psi(f(z))-\Psi(z)=1
\end{equation}
for a parabolic diffeomorphism, called \emph{Abel equation}. Equation \eqref{abel} is used for obtaining \' Ecalle-Voronin moduli of analytic classification, see e.g. \cite{dudko}, \cite{loray}, \cite{sauzin}. There exist sectorially analytic solutions on petals, $\Psi_+$ on $V_+$ and $\Psi_-$ on $V_-$, the so-called \emph{Fatou coordinates}. Their comparison reveals analytic class of $f$. For more details, see Section~\ref{app} or references above.

In this standard situation, the Fatou coordinates $\Psi_+$ and $\Psi_-$ glue to a global Fatou coordinate, analytic in some punctured neighborhood of the origin $U \setminus \{0\}$, if and only if $f$ belongs to the analytic class of $$f_0(z)=\frac{z}{1-z}=\exp(z^2\frac{d}{dz}).$$ That is, if and only if $f$ is a time-one map of a holomorphic vector field. 

In the sequel, we characterize the germs for which the principal parts $H^f$ and $H^{f^{\circ -1}}$ are globally analytic and compare the results with analytic classification results. 
\medskip

We state here the definition of \emph{cohomological difference equations}, which generalize both equations \eqref{ha} and \eqref{abel}. The definition of is known in literature, see e.g. \cite{beli} or \cite{lyu}. Such equations have been mentioned also in \cite[Section A.6]{loraypre}.
\begin{definition}[A cohomological equation for a diffeomorphism $f$]\label{propabelgen}
\emph{A cohomological equation for a diffeomorphism $f$ with the right-hand side $g\in\mathbb{C}\{z\},\ g\equiv\!\!\!\!\!/\ 0$,} is the difference equation
\begin{equation}\label{abelgen}
H(f(z))-H(z)=g(z),
\end{equation}
in some neighborhood of $z=0$.

The function $H$ that satisfies \eqref{abelgen} on some domain is called \emph{a solution of the cohomological equation on the given domain}. 

In particular, if $g=C\cdot \text{Id}^{\,m}$, $C\in\mathbb{C}$, $m\in\mathbb{N}_0$, we call equation \eqref{abelgen} the \emph{$m$-Abel equation}.
\end{definition}
For $g\equiv 1$ we get the Abel equation and for $g(z)=-\pi z$ equation \eqref{ha} for principal parts of areas for $f$. 

In Section~\ref{two}, we discuss solutions of cohomological equations. The results on existence of sectorially analytic solutions are mostly taken from \cite{loraypre}. Our result in Section~\ref{two} is the following Theorem~\ref{glo}. It gives necessary and sufficient conditions on a diffeomorphism $f$ in terms of right-hand side $g$, for the cohomological equation to have a globally analytic solution $H$ in some neighborhood of $0$. That is, its sectorial analytic solutions agree on the components of $V_+\cap V_-$.

Let the right-hand side $g$ of \eqref{abelgen} be \emph{of multiplicity $l$}. That is, 
\begin{equation}\label{mult}
g(z)=\alpha_l z^l+o(z^l)\in \mathbb{C}\{z\},\ \alpha_l\neq 0,\ l\in\mathbb{N}_0.
\end{equation} 
If $l=0$ ($\alpha_0\neq 0$) or $l=1$ ($\alpha_0=0,\ \alpha_1\neq 0$), let us define $$h_{\alpha_0,\alpha_1}(z)=-\frac{\alpha_0}{z}+\alpha_1 \log z.$$

\begin{theorem}[Existence and uniqueness of a globally analytic solution of a cohomological equation]\label{glo}
Let $f$ be a parabolic diffeomorphism formally equivalent to $f_0$. Let $g\in \mathbb{C}\{z\},\ g\equiv\!\!\!\!\!/\ 0,$ be of multiplicity $l\in\mathbb{N}_0$, as in \eqref{mult}. The cohomological equation
$$
H(f(z))-H(z)=g(z)
$$
has a (unique up to a constant) globally analytic solution on some neighborhood of $z=0$ if and only if the diffeomorphism $f$ is of the form$^{\star}$\footnote{$^\star$ If $\alpha_1\neq 0$, then $h_{\alpha_0,\alpha_1}$ contains a logarithmic term and is not a well-defined nor invertible function on some neighborhood of zero. We consider two branches of function $h_{\alpha_0,\alpha_1}$ defined on overlapping sectors, which are invertible. We then glue two sectors together in function $f$ analytic at zero by Riemann theorem on removable singularities. The same holds for function $H$ below.}
\begin{equation}\label{f2}
f(z)=\left\{ \begin{array}{ll}
\varphi^{-1}\bigg(h_{\alpha_0,\alpha_1}^{-1}\Big(h_{\alpha_0,\alpha_1}\big(\varphi(z)\big)+g(z)\Big)\bigg)& ,\ l=0,\ 1,\\[0.2cm]
\varphi^{-1}\left(\varphi(z)\cdot\Big(1+\frac{l-1}{\alpha_l}\frac{g(z)}{\varphi(z)^{l-1}}\Big)^{\frac{1}{l-1}}\right)& ,\ l\in\mathbb{N},\ l\geq 2.
\end{array}\right.
\end{equation}
for some analytic germ $\varphi$, $\varphi(z)\in z+z^2\mathbb{C}\{z\}$.

The globally analytic solution $H$ is then given by 
\begin{equation}\label{nakon}
H(z)=\left\{ \begin{array}{ll}
h_{\alpha_0,\alpha_1}\circ \varphi(z)& ,\ l=0,\ 1,\\[0.2cm]
\frac{\alpha_l}{l-1}\varphi(z)^{l-1}& ,\ l\in\mathbb{N},\ l\geq 2.
\end{array}\right.
\end{equation}
\end{theorem}

Here and in the sequel, we use the term globally analytic in a slightly incorrect manner. In the case where the linear term of $g$ is non-zero, $H$ contains a logarithmic term in the asymptotic expansion, as $z\to 0$. Also, when $g$ contains a constant term, the term $-1/z$ appears in the expansion. Therefore, by globally analytic, we actually mean that $H$ is globally analytic on some neighborhood of $0$, after possibly subtracting the logarithmic term $\log z$ and/or the term $-1/z$. The \emph{global analyticity} of the solution $H$ of \eqref{abelgen} in these cases in fact means the global analyticity of the solution $R$, $H(z)=-\frac{\alpha_0}{z}+\alpha_1 \log z+R(z)$, of the modified equation $$R(f(z))-R(z)=g(z)+\alpha_0\left(\frac{1}{f(z)}-\frac{1}{z}\right)-\alpha_1 \log \big(\frac{f(z)}{z}\big).$$

In Section~\ref{app}, we apply Theorem~\ref{glo} to Abel equation \eqref{abel} to obtain the well-known result about Fatou coordinate being global if and only if $f$ is analitically conjugate to the model $f_0$. Furthermore, we apply Theorem~\ref{glo} to equation~\eqref{ha} for the principal parts of areas. Thus we obtain  Theorem~\ref{pringlo} below. It connects global analyticity property of the principal parts of areas with the intrinsic properties of $f$ and shows that global analyticity is not the \emph{rule}. 
\begin{theorem}[\emph{Global} principal parts of areas]\label{pringlo}
The principal parts\linebreak $(H^f-i\pi^2)$ on $V_+$ and $(\pi\cdot\text{Id}-H^{f^{\circ -1}})$ on $V_-$ glue to a global analytic function on a neighborhood of $z=0$ if and only if the diffeomorphism $f$ is of the form
$$
f(z)=\varphi^{-1}\left(e^z\cdot \varphi(z)\right),
$$ 
for some analytic germ $\varphi(z)\in z+z^2\mathbb{C}\{z\}$. The principal parts are then given by 
\begin{align*}
&H^f(z)=-\pi \log \varphi(z)+i\pi^2-\frac{\pi}{4},\ \ z\in V_+,\\
&H^{f^{\circ -1}}(z)=\pi z+\pi \log \varphi(z)+\frac{\pi}{4},\ \ z\in V_-.
\end{align*}
Here, the branches of complex logarithm are determined by the petals.
\end{theorem}

In Section~\ref{app}, we compare Theorem~\ref{pringlo} with the result about global Fatou coordinate mentioned above. The class of diffeomorphisms with \emph{global principal parts of areas} is different from the class of diffeomorphisms analytically conjugated to $f_0$. In Section~\ref{fourhalf}, we give examples of diffeomorphisms analytically conjugated to the model $f_0$, whose principal parts of areas can be glued globally, as well as those whose principal parts of areas are only sectorially analytic. On the other hand, we give examples of diffeomorphisms with globally analytic principal parts that are not analytically conjugated to $f_0$. This shows that the difference of sectorial solutions of cohomological equation \eqref{ha}, $H_+-H_-$ on components of $V_+\cap V_-$ (equivalently, of the principal parts of areas) is not appropriate for reading the analytic class, as was the case with sectorial Fatou coordinates.
\smallskip

This motivates us to introduce new classifications of parabolic diffeomorphisms with respect to higher-order cohomological equations $$H(f(z))-H(z)=z^m,\ m\geq 1,$$ in Section~\ref{seven}, using the differences of sectorial solutions in a way that mimics the analytic classification obtained from Abel equation. The newly introduced \emph{$m$-conjugacy classes} are described by pairs of analytic germs (up to some identifications) that we call \emph{$m$-moments}. Analytic classes correspond to $0$-moments. This puts our equation \eqref{ha} in a more general context. 

We support our previous observations from Section~\ref{fourhalf} by proving that the analytic classes and the $1$-conjugacy classes are \emph{far away} from each other, in a transversal position. More precisely, we show that any pair of analytic germs can be realized as $1$-moment of some parabolic germ tangent to the identity. Moreover, each $1$-class admits a representative in any analytic class. 

\begin{theorem}[\emph{Transversality} theorem]\label{bijectiveness}\ 
Let $\Phi$ be a mapping that associates to each germ $f$ from model formal class its $1$-moment. The mapping $\Phi$ restricted to any analytic class is \emph{surjective} onto the set of all $1$-moments.
\end{theorem}

In particular, there exist germs in each analytic class with trivial $1$-moments, that is, with globally analytic solutions to equation \eqref{ha}. This gives a negative answer to our question of reading the analytic class from principal parts of areas. However, it opens new prospects of investigating the meaning of new classifications of germs with respect to higher-order moments and of determining the relative position of higher conjugacy classes to each-other.

Precise formulation of Theorem~\ref{bijectiveness} and its proof can be found in Subsection~\ref{sevenone}. The question of injectivity is also addressed in Subsection~\ref{sevenone} in Proposition~\ref{chartriv}.

Finally, in Section~\ref{five}, we give some questions for further research.  

\section{Analyticity of solutions of cohomological equations}\label{two}

By \emph{cohomological equation for a diffeomorphism $f$ with right-hand side $g\in\mathbb{C}\{z\}$}, we mean the equation
\begin{equation}\label{hha}
H(f(z))-H(z)=g(z)
\end{equation}
on some neighborhood of $z=0$, see Definition~\ref{propabelgen} in Section~\ref{one}.  

To understand equation \eqref{hha}, in the following Proposition~\ref{formal} we state results mostly taken and adapted from \cite[Section A.6]{loraypre}. In \cite{loraypre}, the case when $g(z)=O(z^2)$ was treated. Here we adapt it for all $g\in\mathbb{C}\{z\}$. We repeat the steps of the proof, since we need them in the proof of Theorem~\ref{ppart} in Section~\ref{three}. 
\begin{proposition}[Formal and sectorially analytic solutions of cohomological equations, \cite{loraypre}]\label{formal}
Let $g\in\mathbb{C}\{z\},\ g(z)=\alpha_0+\alpha_1 z+\alpha_2 z^2+o(z^2)$, $\alpha_i\in\mathbb{C}$, $i\in\mathbb{N}_0$. There exists a unique formal series solution $\widehat{H}$ of equation~\eqref{hha} without constant term of the form
\begin{equation}\label{form}
\widehat{H}(z)\in-\frac{\alpha_0}{z}+\alpha_1 \log z+z\mathbb{C}[[z]].
\end{equation}
All other formal series solutions in the given scale are obtained by adding an arbitrary constant term.

Furthermore, there exist unique sectorially analytic solutions $H_+$ and $H_-$ without constant term defined on petals $V_+$ and $V_-$ respectively, with 1-Gevrey asymptotic expansion \eqref{form}, as $z\to 0$.
\end{proposition}

\begin{proof}
The proof of existence and uniqueness of the formal solution is straightforward, solving the difference equation \eqref{hha} term by term. 
To prove the existence of sectorially analytic solutions, instead of $H$, we consider the function 
$$
R(z)=H(z)+\frac{\alpha_0}{z}-\alpha_1 \log z.
$$
By \eqref{hha}, $R$ satisfies the difference equation
\begin{equation}\label{r}
R(f(z))-R(z)=\delta(z),
\end{equation}
where $\delta\in z^2\mathbb{C}\{z\}$. Now we directly apply results from \cite[A.6]{loraypre} to find two sectorially analytic functions on petals, $R_+$ on $V_+$ and $R_-$ on $V_-$, that satisfy equation \eqref{r}. Moreover, they admit $\widehat{R}(z)=\widehat{H}(z)+\frac{\alpha_0}{z}-\alpha_1 \log z\in z\mathbb{C}[[z]]$ as the asymptotic expansion, as $z\to 0$. Let us describe shortly the idea of the proof of the existence from \cite{loraypre}.
We consider the following series on $V_+$ and $V_-$ respectively:
\begin{equation}\label{for1}
-\sum_{n\geq 0}\delta\big(f^{\circ n}(z)\big),\ z\in V_+,
\end{equation}
and
\begin{equation}\label{for2}
\sum_{n\geq 1}\delta\big(f^{\circ (-n)}(z)\big),\ z\in V_-.
\end{equation}
It can be proven that the above series converge uniformly on all compact subsets of $V_+$, $V_-$ respectively. Then, by Weierstrass theorem, they converge to analytic functions on petals, which we denote $R_+$ on $V_+$ and $R_-$ on $V_-$:
\begin{align}\label{for}
R_+(z)&=-\sum_{n\geq 0}\delta\big(f^{\circ n}(z)\big),\ z\in V_+,\nonumber\\
R_-(z)&=\sum_{n\geq 1}\delta\big(f^{\circ (-n)}(z)\big),\ z\in V_-.
\end{align}
It can be shown furthermore that both $R_+$ and $R_-$ admit $\widehat{R}$ as their $1$-Gevrey asymptotic expansion on sectors, as $z\to 0$.

The uniqueness of sectorial analytic solutions $R_+$ and $R_-$ on $V_+$ and $V_-$ respectively with the asymptotic expansion $\widehat{R}$ is easy to prove. Iterating equation \eqref{r} along the orbit of $f$, summing the iterations and passing to the limit, it is obvious that any analytic solutions of the type $O(z)$ of \eqref{r} on $V_+$ is necessarily given by the same convergent series \eqref{for1}, and is thus unique. The same can be concluded for $V_-$ and formula \eqref{for2}. 

Finally, the solutions of initial equation \eqref{hha} are given by
\begin{align*}
H_\pm(z)=R_\pm(z)-\frac{\alpha_0}{z}+\alpha_1 \log z \text{ on $V_\pm$},
\end{align*}
where $R_{\pm}$ are as in \eqref{for}.\ On each petal we choose the appropriate branch of logarithm. Using results for $R_\pm$, the analyticity and uniqueness results for $H_\pm$ on $V_\pm$ respectively easily follow. 
\end{proof}

\medskip
We now prove Theorem~\ref{glo} from Section~\ref{one}, about the existence of a global analytic solution $H$ of the cohomological equation \eqref{hha}. In the proof, we need the following proposition, which can be proven easily. Note that the assumptions on the existence of formal expansions are crucial for the implication to hold.
\begin{proposition}\label{seriesana}
Let $\widehat{g}\in\mathbb{C}[[z]]$, and $h\in\mathbb{C}\{z\}$ non-constant. Let $\widehat T\in\mathbb{C}[[z]]$, such that
\begin{equation*}\widehat{T}=h\circ\widehat{g}.\end{equation*} Then $\widehat{T}$ is analytic if and only if $\widehat{g}$ is analytic. 
\end{proposition}

\noindent \emph{Proof of Theorem~\ref{glo}}. We consider two cases separately.

$i)$ $\mathbf{l\geq 2}.$ It is easy to check that the formal solution $\widehat{H}\in z\mathbb{C}[[z]]$ is of the form
$$
\widehat{H}(z)=\frac{\alpha_l}{l-1}z^{l-1}+o(z^{l-1}).
$$
Equivalently, we can write
$$
\widehat{H}(z)=\frac{\alpha_l}{l-1}\widehat{\varphi}(z)^{l-1},
$$
where $\widehat{\varphi}$ is a formal series of the form $z+z^2\mathbb{C}[[z]]$. By Proposition~\ref{seriesana}, $H$ is globally analytic if and only if $\varphi$ is globally analytic. 

Suppose now that $H$ is globally analytic. Putting $H(z)=\frac{\alpha_l}{l-1}\varphi(z)^{l-1}$ in equation \eqref{hha}, we can uniquely express $f$: 
\begin{equation}\label{f1}
f(z)=\varphi^{-1}\left(\Big(\varphi(z)^{l-1}+\frac{l-1}{\alpha_l}g(z)\Big)^{\frac{1}{l-1}}\right).
\end{equation}
Here, $\varphi(z)^{l-1}\sim z^{l-1}$ and $g(z)\sim \alpha_l z^l$, as $z\to 0$. The $(l-1)$-th root we take is uniquely determined, since $f$ and $\varphi$ are tangent to the identity. Formula \eqref{f1} easily transforms to \eqref{f2}.

Conversely, if $f$ is of the form \eqref{f2} for $\varphi\in z+z^2\mathbb{C}\{z\}$, it is easy to see that $H(z)=\frac{\alpha_l}{l-1}\varphi(z)^{l-1}$ satisfies equation \eqref{hha} for $f$ and that the formal expansion is of the form \eqref{form}. By uniqueness in Proposition~\ref{formal}, $H$ is the unique analytic solution of \eqref{hha}.
\smallskip 

$ii)$ $\mathbf{l=0,\ 1}$. Similarly as above, the formal solution $\widehat H$ can be written in the form $$\widehat{H}(z)=h_{\alpha_0,\alpha_1}\circ \widehat{\varphi}(z),$$ where $\widehat{\varphi}\in z+z^2\mathbb{C}[[z]]$. It is easy to see that $\widehat{H}$ can be written as 
$$
\widehat{H}(z)=-\frac{\alpha_0}{z}+\alpha_1 \log z+g\left(\frac{\widehat\varphi (z)-z}{z}\right),
$$ 
where $g$ is a nonconstant analytic germ. Now, by Proposition~\ref{seriesana}, $\widehat{H}$ is \emph{globally analytic} (in the sense $\widehat{H}(z)+\frac{\alpha_0}{z}-\alpha_1 \log z$ is globally analytic) if and only if $\widehat\varphi$ is. We can proceed as in $i)$. 
The function $h_{\alpha_0,\alpha_1}$ in expression \eqref{f2} can be regarded as function with two branches (similarly as logarithmic function). In the case $\alpha_0\neq 0$, it can be regarded as global Fatou coordinate for the time 1-map of the vector field $X_{1,\lambda}$, $\lambda=2\pi i \frac{\alpha_1}{\alpha_0}$, see e.g.\cite{loray}, in the case $\alpha_0=0$, it is merely a logarithmic function. It is then invertible on sectors.  
$\hfill\Box$

\section{Analytic properties of $A^\mathbb{C}(z,\varepsilon)$ in $\varepsilon>0$ and $z\in V_+$}\label{three}

Let us recall the expansion \eqref{asy} in $\varepsilon$ of the directed area of the $\varepsilon$-neighborhood of the orbit $S^f(z)$, $z\in V_+$, for a germ $f$. The formal class of $f$ can be read from the first $k+1$ coefficients which do not depend on the initial point of the orbit, see \cite{resman}. To get some insight about the analytic class, we analyse in following Subsections~\ref{threeone} and \ref{threetwo} the \emph{analytic properties of $A^\mathbb{C}(z,\varepsilon)$ in both parameter $\varepsilon>0$ and variable $z\in V_+$}. We first state its \emph{bad} properties (nonexistence of the full asymptotic expansion and accumulation of singularities in $\varepsilon$ for fixed $z$, nonanalyticity in $z$ for fixed $\varepsilon$). The same can be concluded for orbits of $f^{\circ -1}$ on $V_-$. This explains why we study principal parts of directed areas, $z\mapsto H^f(z)$ and $z\mapsto H^{f^{\circ -1}}(z)$, as the only parts of areas with satisfactory analytic properties.

\subsection{Properties of $\varepsilon\mapsto A^\mathbb{C}(z,\varepsilon)$, $\varepsilon>0$}\label{threeone}\ 

Let $z\in V_+$ be fixed. Let $\varepsilon\mapsto A^\mathbb{C}(z,\varepsilon)$ denote the directed area of the $\varepsilon$-neighborhood of the orbit $S^f(z)$, as function of $\varepsilon\in (0,\varepsilon_0)$.

Proposition~\ref{nonasy} states that the remainder term $R(z,\varepsilon)$ in expansion \eqref{asy} does not have expansion in $\varepsilon$ in a power-logarithm scale after a certain number of terms. This presents an obstacle for extending the function from the positive real line to complex $\varepsilon$, by means of formal series. The proof is in the Appendix.
\begin{proposition}[Nonexistence of full power-logarithmic asymptotic expansion in $\varepsilon$, as $\varepsilon\to 0$]\label{nonasy}
Let $z\in V_+$ be fixed. A full asymptotic expansion of $A^\mathbb{C}(z,\varepsilon)$ in a power-logarithmic scale, as $\varepsilon\to 0$, does not exist. That is, there exists $l\in\mathbb{N}$, such that the remainder term $R(z,\varepsilon)$ in \eqref{asy} is of the form:
\begin{align*}
R(z,\varepsilon)=h_1(z)g_1(\varepsilon)+\ldots+h_{l-1}(\varepsilon) g_{l-1}(\varepsilon)+h(z,\varepsilon),\ h(z,\varepsilon)=O\big(g_l(\varepsilon)\big),\\
\hspace{7cm} \varepsilon\to 0.
\end{align*}
The monomials $g_i(\varepsilon)$ are of power-logarithmic type in $\varepsilon$, of increasing flatness at zero, but the limit
\begin{equation*}
\lim_{\varepsilon\to 0}\frac{h(z,\varepsilon)}{g_l(\varepsilon)}
\end{equation*}
does not exist.
\end{proposition}

Proposition~\ref{accu} expresses an obstacle for the analytic continuation of $A^\mathbb{C}(z,\varepsilon)$ in $\varepsilon$ on the neighborhood of the positive real line. The proof is in the Appendix. 

Let us denote by $z_n=f^{\circ n}(z)$, $n\in\mathbb{N}_0$, the points of the orbit. Let
$d_n=|z_n-z_{n+1}|,\ n\in\mathbb{N}_0$, denote the distances between consecutive points of the orbit and let
\begin{equation*}
\varepsilon_n=\frac{d_n}{2},\ n\in\mathbb{N}_0.
\end{equation*}
Note that $\varepsilon_n\to 0$, as $n\to\infty$.

The loss of regularity of $\varepsilon\mapsto A^\mathbb{C}(z,\varepsilon)$ at points $\varepsilon_n$ at which separation of the tail and the nucleus occurs is related to the different rate of growth of the tail and of the nucleus of $\varepsilon$-neighborhoods in $\varepsilon$, due to their different geometry (overlapping discs in nucleus, disjoint discs in tail).
\begin{proposition}[Accumulation of singularities at $\varepsilon=0$]\label{accu}
Let $\varepsilon_0>0$. The function $\varepsilon\mapsto A^\mathbb{C}(z,\varepsilon)$ is of class $C^1$ on $(0,\varepsilon_0)$ and $C^\infty$ on open subintervals $(\varepsilon_{n+1},\varepsilon_n)$, $n\in\mathbb{N}_0$. However, in $\varepsilon_n$, $n\in\mathbb{N}_0$, the second derivative is unbounded from the right:
$$
\lim_{\varepsilon\to\varepsilon_n-}\frac{d^2}{d\varepsilon^2}A^\mathbb{C}(z,\varepsilon)\in\mathbb{C},\ \lim_{\varepsilon\to\varepsilon_n+}\left|\frac{d^2}{d\varepsilon^2}A^\mathbb{C}(z,\varepsilon)\right|=+\infty.
$$
\end{proposition}

\medskip

\subsection{Properties of $z\mapsto A^\mathbb{C}(z,\varepsilon)$, $z\in V_+$}\label{threetwo}\

Let $\varepsilon>0$ be fixed. The following proposition states that sectorial analyticity cannot be obtained directly considering function $z\mapsto A^\mathbb{C}(z,\varepsilon),\ z\in V_+$, for a fixed $\varepsilon>0$ (similarly on $V_-$). The proof is in the Appendix.

Let $S^{\pm}(\varphi,r),\ \varphi\in(0,\pi),\ r>0,$ denote the (symmetric) sectors of opening $2\varphi$ and radius $r>0$ around any attracting, respectively repelling, direction. 
\begin{proposition}\label{no}
Let $\varepsilon>0$. The function $z\mapsto A^\mathbb{C}(z,\varepsilon)$ is not analytic on any attracting petal $V_+$. The function $z\mapsto A^{\mathbb{C},f^{\circ -1}}(z,\varepsilon)$ is not analytic on any repelling petal $V_-$.
\end{proposition}
Moreover, we show in the proof that $z\mapsto A^\mathbb{C}(z,\varepsilon)$ is not analytic on any open sector $S^+(\varphi,r)\subset V_+,\ r>0,\ \varphi\in(0,\pi)$. Similarly, $z\mapsto A^{\mathbb{C},f^{\circ -1}}(z, \varepsilon)$ is not analytic on any open sector $S^-(\varphi,r)\subset V_-,\ r>0,\ \varphi\in(0,\pi)$. For the proof of Proposition~\ref{no}, see Appendix.

\subsection{Properties of the principal parts of areas}\label{threethree}\

Having described \emph{bad} properties of the directed areas of orbits, we concentrate on their principal parts $z\mapsto H^{f}(z),\ z\in V_+,$ and $z\mapsto H^{f^{\circ -1}}(z),\ z\in V_-$, see Section~\ref{one}. We prove here Theorem~\ref{ppart} about sectorial analyticity of principal parts. 

Before proceeding to the proof, let us note that the relation with the cohomological equation \eqref{ha} in Theorem~\ref{ppart} is inspired by the following Proposition~\ref{eqprin} that follows from the geometry of $\varepsilon$-neighborhoods.
\begin{proposition}\label{eqprin}
The principal parts of areas $H^f$ and $H^{f^{\circ -1}}$ satisfy the following difference equations:
\begin{align}
H^f(f(z))-H^f(z)=-\pi z,\ &z\in V_+,\label{h1}\\
H^{f^{\circ -1}}(f^{\circ -1}(z))-H^{f^{\circ -1}}(z)=-\pi z,\ &z\in V_-.\label{h2}
\end{align}
Here, $V_+$ denotes any attracting and $V_-$ any repelling petal.
\end{proposition}
\begin{proof}
Let us first derive the equation \eqref{h1} for $H^f$. Let $z\in V_+$.
By the definition of the directed area, we have that
\begin{equation}\label{ac}
A^\mathbb{C}(z,\varepsilon)=A^\mathbb{C}(f(z),\varepsilon)+z\cdot \varepsilon^2\pi, \ z\in V_+,
\end{equation}
for $0<\varepsilon<\varepsilon_z$ small enough with respect to $z$.
Putting the expansion \eqref{asy} in \eqref{ac}, we get that
$$
\big[H^f(z)-H^f(f(z))\big]\varepsilon^2+\big(R(z,\varepsilon)-R(f(z),\varepsilon)\big)=\varepsilon^2\pi.
$$
By \eqref{asy}, $R(z,\varepsilon)-R(f(z),\varepsilon)=o(\varepsilon^{2+\frac{1}{k+1}})$. Dividing by $\varepsilon^2$ and passing to the limit as $\varepsilon\to 0$, \eqref{h1} follows. 

Equation \eqref{h2} is derived in the same manner, but considering directed areas of orbits of $f^{\circ -1}$ on the repelling sector $V_-$. 
\end{proof}

The idea of the proof of Theorem~\ref{ppart} is the following. We first derive the expressions for principal parts by analysing coefficients in directed areas. On the other hand, applying the iterative procedure described in Proposition~\ref{formal} to equation \eqref{ha}, we show that the sectorially analytic solutions of equation \eqref{ha} are given by almost the same limit formula as principal parts, up to computable explicit constants. 
\smallskip

\noindent \emph{Proof of Theorem~\ref{ppart}}. 

We analyse the form of the coefficient $H^f(z)$ in front of $\varepsilon^2$ in expansion \eqref{asy}, as function of $z\in V_+$, obtained geometrically. We follow the steps for obtaining the expansions for the tail and of the nucleus from \cite{resman}. Let us remind, the tail of the $\varepsilon$-neighborhood is the part of the $\varepsilon$-neighborhood which is the union of disjoint $\varepsilon$-discs, while the nucleus is the remaining part with overlapping discs. We denote by $z\mapsto H^{f}_N(z),\ z\mapsto H^{f}_T(z),\ z\in V_+,$ the principal parts in the expansions of the directed area of the nucleus and the tail respectively. The following equality holds:
\begin{align}\label{sum}
H^f(z)&=H^{f}_N(z)+H^f_T(z),\ z\in V_+.
\end{align}
Going through the proof of \cite[Lemma 4]{resman}, it can be computed that the principal part for the nucleus is constant and equal to
\begin{equation}\label{nuclei}
H^f_N(z)=-\frac{\pi}{4}(1+\log 4),\ z\in V_+.
\end{equation}
The dependence on $z$ of the principal part comes from the tail. The center of the mass of the tail of the $\varepsilon$-neighborhood of the orbit $S^f(z)$ for $\varepsilon>0$, see \cite[Lemma 5]{resman}, is equal to:
\begin{align*}
\big(A(T_\varepsilon)&\cdot t_{T_\varepsilon}\big)(z)=\varepsilon^2\pi\cdot\sum_{l=0}^{n_\varepsilon}f^{\circ l}(z).
\end{align*}
Here, $n_\varepsilon$ is the index where separation of the tail and the nucleus occurs. Obviously, $n_\varepsilon\to\infty$, as $\varepsilon\to 0$. Expanding the sum above as $n_\varepsilon\to\infty$, we get:
\begin{align}\label{r2}
\big(A(T_\varepsilon)&\cdot t_{T_\varepsilon}\big)(z)=\varepsilon^2\pi\cdot\big(-\log n_\varepsilon+C(z)+o(1)\big),\ \varepsilon\to 0.
\end{align}
Here, $C(z)=c_0\Big(\sum_{l=0}^n f^{\circ l}(z)\Big)$ denotes the constant term in in the asymptotic expansion of $\sum_{l=0}^n f^{\circ l}(z)$, as $n\to\infty$. It is a complex function in the initial point $z$. Let us explain shortly how we get the expansion, as $n\to\infty$, of the above sum $$S(n)=\sum_{l=0}^{n}f^{\circ l}(z).$$ Since $f$ is prenormalized and belongs to the formal class $(k=1,\rho=0)$, we have the formal expansion of $f^{\circ l}(z)$, as $l\to\infty$:
$$
f^{\circ l}(z)=\widehat\Psi^{-1}(l+z)=-l^{-1}+q(z)l^{-2}+o(l^{-2}).
$$
Here, $q(z)$ is complex function of the initial point and $\widehat\Psi$ is formal trivialisation function for $f$, $\widehat\Psi^{-1}\in -1/z+(1/z)^2\mathbb{C}[[1/z]]$, see e.g. \cite{loray} or any standard book on classification of parabolic germs. Putting this expansion in sum $S(n)$, by integral approximation we obtain the expansion 
\begin{equation}\label{sssum}
S(n)=-\log n +C(z)+o(1), \ n\to\infty.
\end{equation} 
For more detail, see the proof of Lemma 5 in \cite{resman}.

We conclude using \eqref{r2} and the expansion for $n_\varepsilon$ in $\varepsilon$ from \cite{resman} that the coefficient in front of $\varepsilon^2$ in expansion \eqref{r2}, as $\varepsilon\to 0$, can be expressed as
\begin{equation}\label{free}
H^f_T(z)=\frac{\pi}{2}\log 2+\pi \cdot c_0\Big(\sum_{l=0}^{n}f^{\circ l}(z)\Big),\ z\in V_+.
\end{equation}
By \eqref{sum}, \eqref{nuclei} and \eqref{free}, we get the expression for the principal part:
\begin{equation}\label{usp}
H^f(z)=-\frac{\pi}{4}+\pi\cdot c_0\Big(\sum_{l=0}^{n}f^{\circ l}(z)\Big),\ z\in V_+.
\end{equation}
Our next step is to prove analyticity of the function $H^f$ given by \eqref{usp} on $V_+$. To this end, we consider the unique analytic solution on $V_+$ without constant term of equation \eqref{ha}:
\begin{equation}\label{g}
H(f(z))-H(z)=-\pi z,
\end{equation}
see Proposition~\ref{formal}. By the proof of Proposition~\ref{formal}, it is given by the limit
\begin{equation}\label{lim2}
H_+(z)=\pi \lim_{n\to\infty}\left(\sum_{l=0}^{n}f^{\circ l}(z)-\log f^{\circ(n+1)}(z)\right),
\end{equation}
which was proven to converge pointwise to an analytic function on $V_+$.

To prove analyticity of $H^f$ on $V_+$, it suffices to show that the expression \eqref{usp} for $H^f(z)$ coincides pointwise with $H_+(z)$ in \eqref{lim2}, up to a constant. For a fixed $z$, by \eqref{sssum}, we estimate the first terms in the asymptotic expansion of $\sum_{l=0}^{n}f^{\circ l}(z)-\log f^{\circ(n+1)}(z)$, as $n\to\infty$:
\begin{align}\label{cc}
\sum_{l=0}^{n}f^{\circ l}&(z)-\log f^{\circ(n+1)}(z)=\\
&=-\log n +c_0\Big(\sum_{l=0}^n f^{\circ l}(z)\Big)+o(1)\nonumber -\log f^{\circ(n+1)}(z)=\\
&=c_0\Big(\sum_{l=0}^n f^{\circ l}(z)\Big)-i\pi+o(1).\nonumber  
\end{align}
The last equality is obtained using the expansion:
\begin{align*}
-&\log^+(f^{\circ (n+1)}(z))-\log n=\\
&\ =-\log^+\left(\varphi_+^{-1}\big(\frac{\varphi_+(z)}{1-(n+1)\varphi_+(z)}\big)\right)-\log n=\\
&\ =-\log^+\left[\left(\frac{\varphi_+(z)\cdot n}{1-(n+1)\varphi_+(z)}\right)\left(1+O\big(\frac{\varphi_+(z)}{1-(n+1)\varphi_+(z)}\big)\right)\right]=\\
&\ =-\log^+\left(\frac{1}{\frac{1-\varphi_+(z)}{n\varphi_+(z)}-1}\right) -\log^-\left(1+O\big(\frac{\varphi_+(z)}{1-(n+1)\varphi_+(z)}\big)\right)=\\
&\hspace{8cm}=\ -i\pi+o(1),\ n\to\infty. 
\end{align*}
Here, $\log^- z$ denotes the principal branch of logarithm for $\arg z\in(-\pi,\pi)$ and $\log^+ z$ the branch for $\arg z\in(0,2\pi)$. The function $\varphi_+$, $\varphi_+(z)=z+a_1z^2+o(z^2)$, denotes the analytic change of variables on $V_+$ that reduces $f$ to its formal normal form $f_0$, $f_0(z)=\frac{z}{1-z}$.

Passing to the limit in \eqref{lim2}, by \eqref{cc}, we get the pointwise equality:
$$
H_+(z)=\pi\cdot c_0\Big(\sum_{k=0}^{n}f^{\circ k}(z)\Big)-i\pi^2,\ z\in V_+.
$$
By \eqref{usp}, we conclude
\begin{equation*}
H^f(z)+\frac{\pi}{4}-i\pi^2=H_+(z).
\end{equation*}
Therefore, since $H_+$ is analytic on $V_+$, $H^f$ is also analytic on $V_+$.

\smallskip

Analyticity of $H^{f^{\circ -1}}$ on $V_-$ can be proven in the same manner, considering inverse diffeomorphism $f^{\circ -1}$, and comparing $H^{f^{\circ -1}}$ with sectorial solution $H_-$ of equation \eqref{g} on $V_-$. By Proposition~\eqref{formal},
\begin{align*}
H_-(z)&=\pi \lim_{n\to\infty}\left(-\sum_{k=1}^{n+1}(f^{\circ -1})^{\circ k}(z)-\log (f^{\circ -1})^{\circ-(n+1)}(z)\right),\ z\in V_-.\nonumber
\end{align*}
$\hfill\Box$

\section{Applications of Theorem~\ref{glo}, global principal parts}\label{app}

\subsection{Application of Theorem~\ref{glo} to the standard Abel equation}\label{fourone}

The trivialization (Abel) equation for a parabolic germ $f$ is central object for describing analytic class of $f$:
\begin{equation}\label{aabel}
\Psi(f(z))-\Psi(z)=1.
\end{equation}
We use here Theorem~\ref{glo} to derive a well-known result by \' Ecalle and Voronin. Of course this is not new, and we put it here only as a trivial example. A parabolic germ $f$ is analytically conjugated to the model $f_0$, $f_0(z)=\frac{z}{1-z}$, if and only if equation \eqref{aabel} has a global (analytic in some punctured neighborhood of the origin) solution $\Psi$. The solution is related to the analytic conjugacy $\varphi$ conjugating $f$ to $f_0$ by $\Psi=\Psi_0\circ\varphi(z),$ where $\Psi_0(z)=-1/z$.

\smallskip

\noindent \emph{Proof by Theorem~\ref{glo}}.
Abel equation is a cohomological equation with the right-hand side $g\equiv 1$. Therefore, $h_{1,0}(z)=-1/z$. By \eqref{f2}, we get that there exists a global analytic solution of \eqref{aabel} if and only if $f$ is given by
$$
f(z)=\varphi^{-1}\left(-\frac{1}{-\frac{1}{\varphi(z)}+1}\right)=\varphi^{-1}\circ f_0\circ \varphi (z),
$$
for some analytic diffeomorphism $\varphi$. It is unique up to additive constant and, by \eqref{nakon}, of the form $\Psi=\Psi_0\circ \varphi.$$\hfill\Box$

\subsection{Nontrivial application of Theorem~\ref{glo} to global analiticity of principal parts of areas. Proof of Theorem~\ref{pringlo}}\label{fourtwo}\

We prove here Theorem~\ref{pringlo} from Section~\ref{one}, which gives the necessary and sufficient conditions on a diffeomorphism $f$ for global analyticity of its principal parts of areas. 
The theorem was motivated by the following. As mentioned in Subsection~\ref{fourone}, the existence of the global solution of the trivialization equation (the global Fatou coordinate) signals that a diffeomorphism is analytically conjugated to the model diffeomorphism $f_0$. On the other hand, when we consider $\varepsilon$-neighborhoods of orbits and their principal parts of areas, the equation 
\begin{equation}\label{haha}
H(f(z))-H(z)=-\pi z
\end{equation}
naturally arises, see Theorem~\ref{ppart} in Section~\ref{one}. This equation looks similar to the trivialization equation. The idea behind Theorem~\ref{pringlo} was to express the existence of its global solution in terms of $f$. 

\medskip
\emph{Proof of Theorem~\ref{pringlo}.} The theorem is a direct consequence of Theorem~\ref{ppart} and Theorem~\ref{glo}. By Theorem~\ref{ppart}, the principal parts of areas are explicitely related to the sectorial solutions of cohomological equation $H(f(z))-H(z)=-\pi z$, with right-hand side $g=-\pi\cdot \text{Id}$. By \eqref{f2} in Theorem~\ref{glo}, this equation has a global analytic solution if and only if $f(z)=\varphi^{-1}(\varphi(z)\cdot e^z)$, for some $\varphi(z)\in z+z^2\mathbb{C}\{z\}$. 
$\hfill\Box$

\smallskip
We give here two simple examples of parabolic germs that posess the \emph{global analyticity property of principal parts} from Theorem~\ref{pringlo}.
\begin{example}[Germs with \emph{global} principal parts] \
\begin{enumerate}
\item[$(1)$] $f(z)=z\cdot e^z$,\ for $\varphi=\text{Id}$, 
\item[$(2)$] $f(z)=-\log(2-e^z)$,\ for $\varphi(z)=1-e^{-z}$.
\end{enumerate}
\end{example}

\section{Counterexamples: trivial analytic class\\ versus global principal parts}\label{fourhalf}
As stated before, we consider only germs belonging to the simplest formal class $(k=1,\rho=0)$, which are prenormalized ($a_1=1)$:
$$
f(z)=z+z^2+z^3+o(z^3).
$$
The first normalizing change already performed, we further admit only changes of variables tangent to the identity. This restriction is rather standard and provides simpler presentation of analytic classification.

The idea behind this article, as mentioned in Section~\ref{one}, was to recover analytic class of a parabolic diffeomorphism by comparing principal parts of areas for $f$ and inverse diffeomorphism $f^{\circ -1}$ on the intersection of petals, or, equivalently, sectorial solutions of cohomological equation \eqref{haha}. 
In this section, we do not solve the classification problem, but present the difficulties that are encountered.

\subsection{\' Ecalle-Voronin moduli of analytic classification}\label{fhh}

We describe here one approach to \' Ecalle-Voronin moduli of analytic classification of germs, a slight reformulation of Fourier representation of analytic moduli from \cite{ecalle,voronin} or \cite{dudko}. The classes are given by pairs of germs at zero, after appropriate identifications. We will use similar approach in Section~\ref{seven} to define new classifications, using higher cohomological equations instead of Abel equation. 

Let $\Psi_+(z),\ z\in V_+,$ and $\Psi_-(z),\ z\in V_-,$ be two sectorial solutions of Abel equation (unique up to additive constant)
\begin{equation}\label{abelstd}
\Psi(f(z))-\Psi(z)=1.
\end{equation}
By $V^{up}$ and $V^{low}$, we denote the upper and the lower component of the intersection $V_+\cap V_-$:
$$V^{up}=\{z\in V_+\cap V_-|Im(z)>0\},\ V^{low}=\{z\in V_+\cap V_-|Im(z)<0\}.$$ 
By pair $\big(h,k\big)$, we denote the differences of sectorial solutions on $V^{up}$, $V^{low}$:
$$
h(z)=\Psi_+(z)-\Psi_-(z),\ z\in V^{up},\quad k(z)=\Psi_-(z)-\Psi_+(z),\ z\in V^{low}.
$$
Under notations from \cite[A.4,A.5]{loraypre}, the pair $(h,k)$ is a \emph{$1$-cocycle}, in the sense that $h$ and $k$ are analytic germs on petals $V^{up},\ V^{low}$, with an exponential decrease, as $z\to 0$. The notion of \emph{$1$-cocycles} will be used more in Section~\ref{seven}.

By equation \eqref{abelstd}, $h$ and $k$ are constant along the closed orbits in $V^{up}$ and $V^{low}$. We choose positive sector for representation of space of orbits. We represent the space of orbits on $V_+$ by punctured sphere, using the change of variables $t=e^{-2\pi i \Psi_+(z)}$ (each orbit corresponds to only one point). Now, closed orbits in $V^{up}$ lift to the punctured neighborhood of the pole $t=\infty$ and closed orbits in $V^{low}$ to the punctured neighborhood of the pole $t=0$. We thus lift $\big(h,k\big)$ to a space of orbits represented by $\Psi_+$ through a pair of germs $t\mapsto \big(g_\infty(t),g_0(t)\big)$ around $t=\infty$ and $t=0$ of punctured sphere:
$$
h(z)=g_\infty(e^{-2\pi i \Psi_+(z)}),\ z\in V^{up};\quad k(z)=g_0(e^{-2\pi i \Psi_+(z)}),\ z\in V^{low}.
$$
Additionally, inverting $g_\infty$, $g_\infty(t)=g_\infty(1/t)$, $g_\infty$ becomes also a germ at $t=0$. It can be seen that the germs are analytic at punctured neighborhood of 0. They can moreover be extended continuously to $0$, by differences of constant terms in sectorial trivialisations. It holds that $g_\infty(0)+g_0(0)=0$. This extension is analytic at $t=0$ by Riemann's characterization of removable singularities. Therefore we get a \emph{pair of analytic germs} $\big(g_\infty,g_0\big)$ at the origin. 

We identify two pairs of germs, $(g_\infty^1,g_0^1)$ and $(g_\infty^2,g_0^2)$ if it holds that:
\begin{align}\label{idii}
g_\infty^1(0)=g_\infty^2(0)+a,&\quad g_0^1(0)=g_0^2(0)-a,\\
g_\infty^1(t)=g_\infty^2(bt),&\quad g_0^1(t)=g_0^2(t/b),\nonumber
\end{align}
for $a\in\mathbb{C}$ and $b\in\mathbb{C}^*$.
This corresponds to choosing trivialisation functions $\Psi_+$ and $\Psi_-$ up to an additive constant.

The \' Ecalle-Voronin classification theorem states that \emph{there exists a bijective correspondence between all analytic classes of diffeomorphisms from the model formal class and all pairs $(g_\infty,g_0)$ of analytic germs at $t=0$ such that $g_\infty(0)+g_0(0)=0$, after identifications \eqref{idii}}. 
\smallskip

The class of diffeomorphisms analytically conjugated to the model $f_0$ is characterized (up to additive constant) by:
\begin{align*}
\Psi_+(z)-\Psi_-(z)\equiv 0,\ z\in V_+\cap V_-.
\end{align*} Thus, it is given by the trivial pair of germs, $(0,0)$, up to identifications \eqref{idii}. That is, by pairs of constant germs of the type $(-a,a)$, $a\in\mathbb{C}$.

\subsection{Subtracting principal parts on intersections of petals, analytic class cannot be read.}\label{fivesecond}\

In our considerations of $\varepsilon$-neighborhoods, the equation
\begin{equation}\label{haj}
H(f(z))-H(z)=-z
\end{equation}
naturally appears instead of Abel equation. We show that considering the difference $H_+(z)-H_-(z)$ on $z\in V_+\cap V_-$ is not sufficient for determining analytic class of $f$, as was the case with Abel equation above

By Riemann's theorem on removable singularities ($R_\pm(z)=H_\pm(z)+\log z$ being bounded on $V_\pm$), the equality of sectorial solutions without constant term of \eqref{haj} on $V^{up}$ and $V^{low}$ (up to a constant from different branches of logarithm):
\begin{equation}\label{differences}
H_+(z)-H_-(z)\equiv 0,\ z\in V^{up};\ \ H_+(z)-H_-(z)\equiv 2\pi i,\ z\in V^{low},
\end{equation}
corresponds to the fact that equation \eqref{haj} has a global analytic solution. 

Let $\mathcal{C}_0$ be the class of diffeomorphisms analytically equivalent to $f_0$. Denote by $\mathcal{S}$ the set of diffeomorphisms where \eqref{haj} has a global analytic solution. By Theorem~\ref{pringlo}, it follows that
$$
\mathcal{S}=\Big\{f(z)=z+z^2+z^3+o(z^3)\Big| f=\varphi^{-1}(e^z\cdot\varphi(z)),\Big.\ \varphi(z)=z+z^2\mathbb{C}\{z\}\Big\}.
$$
The following example shows that the intersection $\mathcal{S}\cap\mathcal{C}_0$ is nonempty. Furthermore, neither of the sets is a subset of another. This pertains to the fact that trivial analytic class and \emph{trivial} class with respect to 1-Abel equation are in general position.
\begin{example}\label{exx1}
\begin{align*}
&f(z)=-\log(2-e^z)\in \mathcal{S}\cap \mathcal{C}_0,\\
&g(z)=z e^z,\ g(z)\in\mathcal{S},\ g(z)\notin\mathcal{C}_0,\\
&f_0(z)\in\mathcal{C}_0,\ f_0(z)\notin\mathcal{S}. 
\end{align*}
In the first example, we take $\varphi^{-1}(z)=-\log(1-z)$ for both classes. The second example follows from the fact that no entire function is analytically conjugated to $f_0$, see \cite{ahern}. The third example follows from Example~\ref{ex1} below.
\end{example}

Since triviality of differences \eqref{differences} is possible for germs analytically conjugated to $f_0$, as well as for those not conjugated, we conclude that information given by the differences $H_+(z)-H_-(z)$ on $V^{up}\cap V^{low}$ is not sufficient for determining the analytic class.

\medskip
In the next example, we compute explicitely the 1-cocycle of differences $H_+(z)-H_-(z),\ z\in V^{up}\cup V^{low},$ for the simplest model diffeomorphism $f_0$. The difference of sectorial trivialisations $\Psi_+(z) -\Psi_-(z)$ was in this case trivial. Here we get a non-trivial cocycle. We apply the method of Borel-Laplace summation directly to the difference equation. The procedure is standard and a similar one can be found in e.g. \cite[Example 2]{dudko} or \cite{candelpergher}.
\begin{example}[The differences for the model germ $f_0(z)=\frac{z}{1-z}$]\label{ex1}
We substitute $\widehat{H}=-Log+\widehat{R},\ \widehat{R}\in z\mathbb{C}[[z]],$ in the equation \eqref{haj} for $f_0$ and thus obtain the equation for $\widehat{R}$:
\begin{equation*}
\widehat{R}(f_0(z))-\widehat{R}(z)=-z+Log\frac{f_0(z)}{z}.
\end{equation*}
By the change of variables $w=-1/z$, denoting $\widehat{\widetilde R}(w)=\widehat{R}\circ \chi,\ \chi(w)=-1/w$, we get
\begin{align}\label{er}
\widehat{\widetilde{R}}(w+1)-\widehat{\widetilde{R}}(w)&= w^{-1}- Log(1+w^{-1})=\sum_{k=2}^{\infty}(-1)^{k}\frac{w^{-k}}{k}.
\end{align}
The right-hand side of this equation is of the type $w^{-2}\mathbb{C}\{w^{-1}\}$. We denote it by $$b(w)=\sum_{k=2}^{\infty}\frac{(-1)^k}{k}w^{-k}.$$  

Applying the Borel transform to \eqref{er}, we get
\begin{equation*}
\mathcal{B}\widehat{\widetilde{R}}(\xi)=\frac{\mathcal{B}b(\xi)}{e^{-\xi}-1}, \ \ \mathcal{B}b(\xi)=\frac{e^{-\xi}+\xi-1}{\xi}.
\end{equation*}
It can be shown that the function $\xi\mapsto\mathcal{B}\widehat{\widetilde{R}}(\xi)$ has $1$-poles at $2i\pi\mathbb{Z}^*$ in directions $\theta=\pm\pi/2$, and it is exponentially bounded and analytic in every other direction. For details, see \cite{dudko}.
Therefore, $\widehat{\widetilde R}$ is 1-summable in the arcs of directions $I_+=(-\pi/2,\pi/2)$ and $I_-=(\pi/2,3\pi/2)$. The Laplace transform yields two analytic solutions as $1$-sums, $\widetilde R^+$ on $W_+=\{w\ |\ Re(w e^{i\theta})>\beta_0,\ \theta\in I_+\}$, and $\widetilde R^{-}$ on $W_-=\{w\ |\ Re(w e^{i\theta})>\beta_0,\ \theta\in I_-\}$, where $\beta_0>0$ is some constant. By the residue theorem applied to the difference of Laplace integrals, on intersections of $W_+$ and $W_-$ they differ by $1$-periodic functions. For $w\in W^{up}=\{w|\ Im(w)>\beta_0\}$, we have:
\begin{align*}
\widetilde{R}^+(w)-\widetilde{R}^-(w)&=\int_{0}^{\infty\cdot e^{i\theta_1}} \frac{e^{-\xi w}\mathcal{B}b(\xi)}{e^{-\xi}-1}d\xi  -\int_{0}^{\infty\cdot e^{i\theta_2}}\frac{e^{-\xi w}\mathcal{B}b(\xi)}{e^{-\xi}-1}d\xi=\\
&=\int_{\infty\cdot e^{i\theta_2}}^{\infty\cdot e^{i\theta_1}} \frac{e^{-\xi w}\mathcal{B}b(\xi)}{e^{-\xi}-1}d\xi =\\
&=2\pi i\cdot\sum_{k=1}^{\infty} Res(\frac{e^{-\xi w}\mathcal{B}b(\xi)}{e^{-\xi}-1},\xi=-2\pi i k)=\\
&=-2\pi i\sum_{k\in\mathbb{N}}e^{2\pi i k\cdot w}=-2\pi i\frac{e^{2\pi i \cdot w}}{1-e^{2\pi i \cdot w}}.
\end{align*}
Here, $\theta_2\in(-\pi/2,\pi/2)$ and $\theta_1\in(\pi/2,3\pi/2)$ are close to $-\pi/2$. 

Similarly, for $w\in W^{low}=\{w|\ Im(w)<-\beta_0\}$, we get
\begin{align*}
\widetilde{R}^+(w)-\widetilde{R}^-(w)&=2\pi i\frac{e^{-2\pi i \cdot w}}{1-e^{-2\pi i \cdot w}}.
\end{align*}
\noindent Replacing $\widehat{\widetilde R}$ with $\widehat{\widetilde H}$ and returning to the variable $z=-1/w$, we get
\begin{align*}
H_+(z)-H_-(z)&=-2\pi i \frac{e^{-2\pi i \frac{1}{z}}}{1-e^{-2\pi i \frac{1}{z}}}=-2\pi i f_0(e^{-2\pi i \frac{1}{z}}),\ z\in V^{up},\\[0.1cm]
H_-(z)-H_+(z)&=-2\pi i-2\pi i \frac{e^{2\pi i \cdot \frac{1}{z}}}{1-e^{2\pi i \cdot \frac{1}{z}}}=-2\pi i-2\pi i f_0(e^{2\pi i \cdot \frac{1}{z}}),\\
&\hspace{7cm} z\in V^{low}.\nonumber
\end{align*}
Here, $V_+$ and $V_-$ are petals in the $z$-plane, obtained by inverting $W_+$ and $W_-$ by $z=-1/w$, and $V^{up}$ and $V^{low}$ are their intersections.
\end{example}

We see that for the model $f_0$, the \emph{1-cocycle} of differences $H_+-H_-$ lifted to orbit space is exactly the germ $-2\pi i f_0(t)$ itself, in both components. See Section~\ref{seven} for details. This is certainly not a coincidence. It would be interesting to have some geometrical explanation.

In the above manner, the differences can be computed by Borel-Laplace transform for any germ $f$ analytically conjugated to $f_0$, and it can be seen in general that the cocycles are not trivial. 
\smallskip
\begin{example}[Explicit formulas for the sectorial solutions $H^{f_0}_{\pm}$ of $1$-Abel equation for the model $f_0$]\label{ex2}
By \eqref{usp} in the proof of Theorem~\ref{ppart}, we get:
\begin{align*}
H^{f_0}_+(z)&=\pi\cdot c_0\Big(\sum_{k=0}^{n}\frac{z}{1-kz}\Big)-i\pi^2=\pi\cdot \frac{d}{dz}\log (\Gamma(z))\Big|_{-\frac{1}{z}}\Big.-i\pi^2,\ z\in V_+,\\
H^{f_0}_-(z)&=\pi z-\pi\cdot c_0\Big(\sum_{k=0}^{n}\frac{z}{1+kz}\Big)=\pi z+\pi\cdot \frac{d}{dz}\log (\Gamma(z))\Big|_{\frac{1}{z}}\Big.,\ z\in V_-.
\end{align*}
Here, $\Gamma$ is the standard Gamma function, holomorphic on $\mathbb{C}\setminus -\mathbb{N}_0$. Therefore, $H^{f_0}_{\pm}$ are well-defined and analytic on $V_\pm$.
\end{example}

\section{Higher-order moments and higher conjugacy classes}\label{seven}
In this section, we put the difference of solutions of Abel and of $1$-Abel equation in a more general context. We have anticipated in Theorem~\ref{glo} and supported in Examples \ref{exx1} and \ref{ex1} the fact that the trivial class for $1$-Abel equation \eqref{haj} is not related to the trivial analytic class. 

Nevertheless, subtracting the sectorial solutions $H_+-H_-$ of \eqref{haj}, we define another classification of parabolic diffeomorphisms, different from the analytic classification. Further classifications can be derived by comparing sectorial solutions of higher-order cohomological equations. 

It would be interesting to analyse the relative position of classes defined by equations of different orders. For our problem, the relative position of analytic classes and classes with respect to the equation \eqref{haj} is most important, to see how \emph{far away} they actually are from each other. Here we define higher-order classes and give some results about their positions.

Let
\begin{equation}\label{korder}
H(f(z))-H(z)=-z^m
\end{equation}
be the \emph{$m$-Abel equation} for germ $f$, $m\in \mathbb{N}_0$, see Definition~\ref{propabelgen} in Section~\ref{one}. By Proposition~\ref{formal}, there exist unique up to constant analytic solutions $H_+^m$ and $H_-^m$ of \eqref{korder} on petals $V_+$ and $V_-$. Subtracting \eqref{korder} for $H_+^m$ and $H_-^m$, we easily get that
$
H_+^m-H_-^m$ is constant along the closed orbits in $V^{up}$ and $V^{low}$, as was the case in Subsection~\ref{fhh}.
  
We now mimic the procedure described in Subsection~\ref{fhh} to define new classifications imposed by higher-order Abel equations. We can lift the exponentially decaying $1$-cocycle $(h,k)$ on $V^{up}\times V^{low}$ to space of orbits of both sectors by composition with exponential function:
\begin{align}\label{baz} 
h(z)=H_+^m(z)-H_-^m(z)=&g_\infty^{m,+}(e^{-2\pi i \Psi_+(z)})=g_\infty^{m,-}(e^{-2\pi i \Psi_-(z)}),\ z\in V^{up},\nonumber\\[0.1cm]
k(z)=H_-^m(z)-H_+^m(z)=&(-2\pi i)+ g_0^{m,+}(e^{-2\pi i \Psi_+(z)})=\nonumber\\
=&(-2\pi i) + g_0^{m,-}(e^{-2\pi i \Psi_-(z)}),\ \hfill z\in V^{low}.
\end{align}
The term $2\pi i$ is put in brackets, since it appears only in the case when $m=1$, due to different branches of logarithm.

For representation of functions defined on orbit spaces, in the sequel we always choose trivialization function $\Psi_+$ of the attracting sector. Therefore, we work only with functions $t\mapsto g^{m,+}_{\infty}(t)$ and $t\mapsto g^{m,+}_0(t)$ on the neighborhoods of poles $t=\infty$ and $t=0$ of a punctured sphere, and denote them simply by $g^m_0$, $g^m_\infty$. We invert $g^m_\infty(t)=g^m_\infty(\frac{1}{t})$ to obtain two analytic germs at zero. Both germs can be extended analytically to $t=0$, see \cite{loraypre}.  

Note that the trivialisation function $\Psi_+$ is determined only up to an arbitrary constant. Also, if we add any complex constant to $H_+^m$ or $H_-^m$, they remain the solutions of $1$-Abel equation \eqref{haj}. As before, due to this freedom of choice, we identify two pairs of germs $(g_\infty^1,g_0^1)$ and $(g_\infty^1,g_0^1)$ if \eqref{idii} holds.

Note that we can always suppose that $g_0^m(0)+g_\infty^m(0)=0$. The constant term in germs is the difference of constant terms in solutions $H_+^m$ and $H_-^m$, so we choose the solutions without constant term.

With all the notations as above, we define
\begin{definition}[$m$-moments for diffeomorphisms]\label{momentdefinition}
Let $m\in\mathbb{N}_0$. The \emph{$m$-moment of a diffeomorphism $f$ with respect to trivialization function of the attracting petal} or, shortly, \emph{$m$-moment of $f$}, is the pair $$\big(g^m_\infty,\ g^m_0\big)$$ of analytic germs at zero from \eqref{baz}, up to identifications \eqref{idii}.
\end{definition}
Note that the germs are not necessarily diffeomorphisms.

\begin{remark}\label{ovaj}
In the case of $1$-Abel equation, the $1$-moments are in fact defined by subtracting the sectorial solutions $R_+-R_-$, $z\in V^{up}\cap V^{low}$, of the modified equation 
\begin{equation*}
R(f(z))-R(z)=-z+\log\Big(\frac{f(z)}{z}\Big),
\end{equation*}
instead of sectorial solutions $H_+-H_-$ of the original $1$-Abel equation \eqref{haj}. Here, $H(z)=-\log z+R(z)$. Thus we remove the constant term $-2\pi i$ in \eqref{baz}, coming from different branches of logarithm. 
\end{remark}
\medskip

We now divide the germs of formal type $(k=1,\lambda=0)$ into equivalence classes, putting those which share the same $m$-moment inside the same class.
\begin{definition}[The $m$-conjugacy relation for parabolic germs]
Let $m\in\mathbb{N}_0$. The \emph{$m$-conjugacy} is the equivalence relation on the set of all germs from the model formal class, given by
$$
f_1\stackrel{m}\sim f_2 \text{, if and only if $f$ and $g$ have the same $m$-moments}.
$$
By $[f]_m=\{g\ |\ g\stackrel{m}\sim f\}$
we denote the equivalence class of $f$ with respect to $m$-conjugacy.
\end{definition}

\medskip
\noindent We illustrate the definition on the two most important examples for this work.
\begin{example}[$0$- and $1$-conjugacy classes]\ 
\begin{enumerate}
\item The $0$-Abel equation is in fact Abel equation.  The $0$-moments correspond to \' Ecalle-Voronin moduli, as described in Subsection~\ref{fhh}. The $0$-conjugacy classes correspond to standard analytic classes. In particular, the germs analytically conjugated to the model $f_0$ have trivial $0$-moment, the pair $(0,0)$ $($the Abel equation has globally analytic solution$)$.
\item The $1$-conjugacy classes are obtained using $1$-Abel equation \eqref{haj}. By Theorem~\ref{glo}, the trivial $1$-conjugacy class $($the set of all germs with $1$-moments equal to $(0,0)$, that is, the set of all germs with globally analytic solution to equation \eqref{haj}$)$ is the set $$\mathcal{S}=\Big\{f\ |\ f=\varphi^{-1}(e^z\cdot\varphi(z)),\Big.\ \varphi(z)=z+z^2\mathbb{C}\{z\}\Big\}.$$
\end{enumerate}
\end{example}
\medskip

We complete the section with converse question of \emph{realization} of $0$-moments and $1$-moments. The question is important since it states that all possible $0$- or $1$-conjugacy classes may be represented by all possible pairs of analytic germs $(g_1,g_2)$ at zero, such that $g_1(0)+g_2(0)=0$, up to identifications \eqref{idii}.
 
\begin{proposition}[Realization of $0$-moments]\label{remi} For every pair $(g_1,g_2)$ of analytic germs at zero, such that $g_1(0)+g_2(0)=0$, there exists a germ $f$ from model formal class, such that the pair $\big(g_1,g_2\big)$ is realised as is its $0$-moment. 
\end{proposition}  
\begin{proof}
The $0$-moments are in fact \' Ecalle-Voronin moduli, and the statement follows directly from the theorem of realization of \' Ecalle-Voronin moduli, see \cite{ecalle,voronin} or \cite[Theorem 18]{dudko}.
\end{proof}

\begin{proposition}[Realization of $1$-moments]\label{realis}
For every pair $(g_1,g_2)$ of analytic germs at zero, such that $g_1(0)+g_2(0)=0$, there exists a diffeomorphism $f$ from model formal class, such that the pair $\big(g_1,g_2\big)$ is realised as its $1$-moment. 
\end{proposition}
Note that by varying constant term chosen in sectorial trivialisation function $\Psi_+$ and constants chosen in solutions $H_+$ and $H_-$, we can realise all other $1$-moments identified by \eqref{idii} by the same germ $f$. 
\begin{proof}
This proposition is proven in the proof of Theorem~\ref{bijectiveness} in Subsection~\ref{sevenone} below.
\end{proof}

\subsection{Relative position of $1$-conjugacy and analytic classes.}\label{sevenone}\

It was noted in Example~\ref{ex2} in Subsection~\ref{fivesecond} that there exists no inclusion relation between the trivial analytic class and the trivial $1$-conjugacy class. We investigate here the relative position of analytic classes and $1$-conjugacy classes and prove that they lie in transversal position, that is, they are far away and not related to each other. In this way we explain and support counterexamples from Subsection~\ref{fivesecond}, that claimed that analytic class cannot be read from the differences of sectorial solutions of $1$-Abel equation. On the other hand, they can be read from differences of sectorial solutions of Abel equation. 

On the other hand, the relative positions of higher conjugacy classes to each other are not discussed and remain the subject for further research.
\medskip

Let $\Phi$ denote the mapping
$$
\Phi(f)=[f]_1,
$$
defined on the set of all germs from model formal class. It attributes to each diffeomorphism its $1$-conjugacy class. By Proposition~\ref{realis}, the $1$-conjugacy classes can equivalently be represented by all pairs of analytic germs $(g_1,g_2)$, $g_1(0)+g_2(0)=0$, up to identifications \eqref{idii}.
\bigskip

We precisely state and prove Theorem~\ref{bijectiveness} from Section~\ref{one}:
\begin{theoreme}[\emph{Transversality} of analytic and $1$-conjugacy classes]
Let $[f]_0$ denote any analytic class. The restriction $\Phi\Big|_{[f]_0}\Big.$ maps surjectively from $[f]_0$ onto the set of all $1$-conjugacy classes. 
\end{theoreme}

We first give an outline of the proof. We first prove the statement from Proposition~\ref{realis} that every pair of germs $\big(g_1,g_2\big)$ such that $g_1(0)+g_2(0)=0$ can be realized as $1$-moment of some germ from model formal class. Moreover, we prove that each $1$-conjugacy class has its representative inside any analytic class. 

Take any analytic class $[f]_0$ and any representative $f$. Let $(g_1,g_2)$ be any pair of analytic germs, satisfying $g_1(0)+g_2(0)=0$. We will show that there exists a diffeomorphism $g\in[f]_0$ whose $1$-moment is equal to $(g_1,g_2)$. We show that there exists an analytic, tangent to the identity right-hand side $\delta$ of the cohomological equation for $f$, such that the pair $(g_1,g_2)$ represents the moment of $f$ with respect to this equation. The idea for first part is \emph{borrowed} from \cite[A.6]{loraypre}. Then, simply by change of variables, we transform the equation to $1$-Abel equation, but for a different diffeomorphism. This new diffeomorphism is analytically conjugated to $f$ by $\delta$.
\smallskip

\emph{Proof of Proposition~\ref{realis} and of Theorem~\ref{bijectiveness}.} Let $[f]_0$ be any analytic class and $f\in [f]_0$ any representative. Moreover, let $\Psi_+^f(z)$ be any trivialisation of the attracting sector $V_+$ for $f$. Let $(g_1,g_2)$ be any pair of analytic germs, satisfying $g_1(0)+g_2(0)=0$.

On some petals $V^{up}$ and $V^{low}$ of opening $\pi$ and centered at directions $\pm i$ respectively,
we define the pair $\big(T_\infty,T_0\big)$ by:
\begin{align}\label{jen}
T_\infty(z)&=g_1(e^{2\pi i \Psi_+^{f}(z)}),\ z\in V^{up},\nonumber\\
T_0(z)&=g_2(e^{-2\pi i \Psi_+^{f}(z)}),\ z\in V^{low}.
\end{align}
If $g_1(0),\ g_2(0)\neq 0$, we first subtract the constant term. This can be done without loss of generality, since a constant term can be added to any sectorial solution afterwards. Also note that $T_\infty$ and $T_0$ are $f$-invariant by construction.

The functions $z\mapsto T_0(z)$ and $z\mapsto T_\infty(z)$ are obviously analytic and exponentially decreasing of order one on $V^{up}$ and $V^{low}$. Therefore, the pair $(T_\infty,T_0)$ defines a $1$-cocycle in the sense from \cite[A.6]{loraypre}. By Ramis-Sibuya theorem, see e.g. \cite[Th\' eor\` eme]{loraypre}, there exists  $1$-summable formal series $\widehat{H}\in z\mathbb{C}[[z]]$, whose differences of $1$-sums, $H_+$ on $V_+$ and $H_-$ on $V_-$, realize the cocycle $(T_\infty,T_0)$. That is,
\begin{equation}\label{dva}
T_0(z)=H_+(z)-H_-(z) \text { on $V^{up}$},\quad T_\infty(z)=H_-(z)-H_+(z) \text{ on $V^{low}$}.
\end{equation}

We adapt now slightly functions $H_+$ and $H_-$ by adding the appropriate branch of logarithm,
\begin{equation}\label{adapt}
\widetilde{H}_+(z)=-H_+(z)+\log(z),\ z\in V_+;\ \ \widetilde{H}_-(z)=-H_-(z)+\log(z),\ z\in V_-.
\end{equation}
We define functions $\delta_\pm$ on $V_{\pm}$ respectively by:
\begin{align}\label{kasn}
\delta_+(z)&=\widetilde{H}_+(f(z))-\widetilde{H}_+(z),\ z\in V_+,\nonumber\\ 
\delta_-(z)&=\widetilde{H}_-(f(z))-\widetilde{H}_-(z),\ z\in V_-. 
\end{align}
From \eqref{dva} and \eqref{kasn}, using $f$-invariance of $T_\infty$ and $T_0$ and Riemann's theorem on removable singularities, we see that $\delta_+$ and $\delta_-$ glue to an analytic germ $\delta$. By \eqref{adapt}, $\delta(z)\in z+z^2\mathbb{C}\{z\}$. 

To conclude, $\widetilde{H}_+$ and $\widetilde{H}_-$ are sectorial solutions of the cohomological equation for diffeomorphism $f$, with the right-hand side $\delta$. That is,
$$
\widetilde{H}(f(z))-\widetilde{H}(z)=\delta(z).
$$
By analytic change of variables $w=\delta(z)$ and multiplying by $(-1)$, we get
$$
-\widetilde H\circ\delta^{-1}(\delta\circ f\circ \delta^{-1}(w))-(-\widetilde H\circ\delta^{-1})(w)=-w.
$$
Therefore, $-(\widetilde{H}\circ\delta^{-1})_{\pm}(z)=-(\widetilde{H}_{\pm}\circ\delta^{-1})(z)$, $z\in V_\pm$ ($V_\pm$ being in fact $\delta(V_\pm)$, but identified with $V_\pm$ since $\delta$ is a conformal map tangent to the identity) are solutions of $1$-Abel equation for diffeomorphism $g=\delta\circ f\circ \delta^{-1}$, analytically conjugated to $f$. The former equality on petals holds by formulas from Proposition~\ref{formal} applied to both cohomological equations, since $\delta$ is an analytic change of variables. Furthermore, by \eqref{jen} and \eqref{dva},
\begin{align*}
-(\widetilde H_+&\circ\delta^{-1})(z)+(\widetilde H_-\circ\delta^{-1})(z)=T_\infty(\delta^{-1}(z))=\\
&=g_1(e^{2\pi i \Psi_+^{f}\circ \delta^{-1}(z)})=g_1(e^{2\pi i \Psi_+^{g}}(z)),\ z\in V^{up},\\
-(\widetilde H_-&\circ\delta^{-1})(z)+(\widetilde H_+\circ\delta^{-1})(z)=-2\pi i+T_0(\delta^{-1}(z))=\\
&=-2\pi i+g_2(e^{-2\pi i \Psi_+^{f}\circ \delta^{-1}(z)})=-2\pi i+g_2(e^{-2\pi i \Psi_+^{g}}(z)),\ z\in V^{low}.
\end{align*}
Here, $\Psi_+^{g}(z)=\Psi_+^f\circ\delta^{-1}$ is a trivialisation function for $g$, for an appropriate choice of constant term, see Remark~\ref{idiot} below.
\smallskip

Thus, the cocycle $(g_1,g_2)$ is realized as $1$-moment of the diffeomorphism $g$, analytically conjugated to $f$. 
$\hfill\Box$

\bigskip
We pose the question of \emph{injectivity} in Theorem~\ref{bijectiveness}. That is, if inside each analytic class there exist different diffeomorphisms with the same $1$-moments. We show in the next Proposition~\ref{formclas} that the injectivity is not true. Inside the trivial analytic class, we even characterize the diffeomorphisms that have the same $1$-moments in Proposition~\ref{chartriv}. 

\begin{proposition}[Non-injectivity]\label{formclas}
Let $[f]_0$ be any analytic class. Let $f,\ g\in [f]_0$. If there exists a change of variables $\varphi\in z+z^2\mathbb{C}\{z\}$ conjugating $f$ to $g$, $g=\varphi^{-1}\circ f\circ \varphi$, of the form
\begin{equation}\label{classi}
\varphi^{-1}=\text{Id}+r\circ f-r,
\end{equation}
where $r\in\mathbb{C}\{z\}$ is analytic, then $f$ and $g$ have the same $1$-moments. 
\end{proposition}
Note that for $f$ and $g$ belonging to the same analytic class, their formal conjugacy is not unique. See Remark~\ref{idiot} for understanding of all formal changes conjugating $g$ to $f$. Note that at least one formal change conjugating $g$ to $f$ is analytic, but \emph{not every other formal change of variables is necessarily analytic}. 
\medskip

In the trivial analytic class, we can get even stronger equivalence statement:
\begin{proposition}[Characterization of germs in trivial analytic class with the same $1$-moments]\label{chartriv}
Let $f,\ g\in [f_0]_0$ be analytically conjugated to $f_0$. $f$ and $g$ have the same $1$-moments if and only if there exists a change of variables $\varphi\in z+\mathbb{C}\{z\}$ conjugating $f$ to $g$ of the form
\begin{equation*}
\varphi^{-1}=\text{Id}+r\circ f-r,
\end{equation*}
where $r\in\mathbb{C}\{z\}$ is analytic.
\end{proposition}

\begin{remark}[About Propositions~\ref{formclas} and \ref{chartriv}]\
\begin{enumerate}
\item The accent in the propositions is on $r$ being \emph{globally} analytic. Indeed, for any change of variables $\varphi^{-1}(z)$, there exists a sectorially analytic function $r$ such that \eqref{classi} holds, since it can be rewritten as the cohomological equation \begin{equation}\label{ti}r(f(z))-r(z)=\varphi^{-1}(z)-z.\end{equation}  However, \emph{good} changes are only those $\varphi$ for which equation \eqref{ti} with right-hand side $(\varphi^{-1}-\text{Id})$ has globally analytic solution.
\smallskip

\item The propositions are \emph{constructive}. Using \eqref{classi}, for every germ $f$ we can construct infinitely many germs inside its analytic class, such that they all belong to the same $1$-conjugacy class.
\smallskip

\item The question remains if Proposition~\ref{chartriv} is true in all analytic classes, not only in trivial class. There seems to be a technical obstacle in the proof, which we do not know how to bypass.
\end{enumerate}
\end{remark}

\smallskip
For the proof of propositions, we need the following known result from e.g. \cite{ilya}. If two germs are formally conjugated, their formal conjugacy is \emph{not} unique, along the same lines as their formal trivialisations are unique only up to an additive constant. Remark~\ref{idiot} provides the description of all formal conjugacies between them.
\begin{remark}[Non-uniqueness of formal conjugation of germs, Theorem 21.12 from \cite{ilya}]\label{idiot}
Let $f$ be formally conjugated to $f_0$. The formal conjugacy $\widehat\varphi$ is unique up to precomposition by germs $f_c\in z+z^2\mathbb{C}\{z\}$ of the form $$f_c(z)=\frac{z}{1-cz},\quad c\in\mathbb{C}.$$ This sole freedom of choice corresponds to adding a constant term $c\in\mathbb{C}$ in trivialisation series $\widehat\Psi^f$ of $f$, related to the conjugacy $\widehat\varphi$ by $\widehat\Psi^f=\Psi_0\circ\widehat\varphi,\ \Psi_0(z)=-1/z$. 

Let two germs $f$ and $g$ be formally conjugated. For any choice of their formal trivialisations $\widehat{\Psi}^f(z)$ and $\widehat{\Psi}^g(z)$ (that is, for any choice of constant terms), there exists a formal conjugation $\widehat\varphi\in z+z^2\mathbb{C}[[z]]$ conjugating $f$ to $g$, such that  
\begin{equation}\label{sto}
g=\widehat\varphi^{-1}\circ f\circ \widehat\varphi,\text{ \ and \ \ }\widehat\Psi^g=\widehat \Psi^f\circ \widehat \varphi.
\end{equation}
Also, for any formal conjugation $\widehat\varphi(z)\in z+z^2\mathbb{C}[[z]]$ there exist trivialisations  $\widehat{\Psi}^f(z)$ and $\widehat{\Psi}^g(z)$ such that \eqref{sto} holds. All possible choices of constants in trivialisation series result in all possible conjugacies $\widehat \varphi(z)$ conjugating $f$ and $g$.
\end{remark}

\noindent \emph{Proof of Propositions~\ref{formclas} and \ref{chartriv}}.\

We first prove the implication of Proposition~\ref{chartriv} that holds only for germs in the trivial analytic class. Let $f$ and $g$ be two germs analytically conjugated to $f_0$. By the first part of Remark~\ref{idiot}, we conclude that any formal conjugacy between $f$ and $g$  is necessarily analytic. This property of trivial analytic class that is not satisfied for other analytic classes. Due to this, we cannot carry out the same proof for other analytic classes. Suppose that $f$ and $g$ have the same $1$-moment, $(g_1,g_2)$. Since $1$-moments are determined only up to identifications \eqref{idii}, we can choose constants in trivialisations $\Psi^f$ and $\Psi^g$ such that the moments are exactly the same. Here, we neglect the possible constant term in $1$-moments, simply choosing the same constant term in $H_+^f$ and $H_-^f$ and $H_+^g$ and $H_-^g$.
Put $R_\pm^{f,g}(z)=H_\pm^{f,g}(z)+\log(z)$, as in Remark~\ref{ovaj}. Then
\begin{align*}
&R^f_+(z)-R^f_-(z)=g_1(e^{2\pi i\Psi^f(z)}),\\
&R^g_+(z)-R^g_-(z)=g_1(e^{2\pi i\Psi^g(z)}),\ z\in V^{up}.
\end{align*}
The same holds with $g_2$ on $V^{low}$. By Remark~\ref{idiot}, for the choice of trivialisations $\Psi^f$ and $\Psi^g$, there exists an analytic change of variables tangent to the identity $\varphi$, such that $\Psi^g=\Psi^f\circ \varphi$ and $g=\varphi^{-1}\circ f\circ \varphi$. We therefore get
\begin{align*}
&R^f_+(z)-R^f_-(z)=g_1(e^{2\pi i\Psi^f(z)}),\\
&R^g_+\circ \varphi^{-1}(z)-R^g_-\circ \varphi^{-1}(z)=g_1(e^{2\pi i\Psi^f(z)}),\ z\in V^{up}.
\end{align*}
Similarly on $V^{low}$. The two formal series $\widehat{R}^f$ and $\widehat{R}^g\circ \varphi^{-1}$ thus realize the same $1$-cocycle on $V^{up}\times V^{low}$. By Ramis-Sibuya theorem, see \cite[Th\' eor\` eme]{loraypre}, they can differ only by converging series $r_1\in \mathbb{C}\{z\}$,
\begin{equation*}
\widehat{R}^g\circ \varphi^{-1}=\widehat{R}^f+r_1.
\end{equation*}
We then have 
\begin{equation}\label{srr}
\widehat{H}^g\circ \varphi^{-1}=\widehat{H}^f+r,
\end{equation}
for $r(z)=r_1(z)-\log(\varphi^{-1}(z)/z),\ r\in\mathbb{C}\{z\}$.
\medskip

\noindent Putting \eqref{srr} in equation $\widehat H^g\circ\varphi^{-1}\circ f-\widehat H^g\circ\varphi^{-1}=-\varphi^{-1}$, obtained from $1$-Abel equation for $g$ after change of variables, we finally get
$$
-id=\widehat H^f\circ f-\widehat H^f=-\varphi^{-1}-r\circ f+r.
$$

We now prove the converse for diffeomorphisms in any analytic class. Let $f$ and $g$ belong to any analytic class, $f,\ g\in[f]_0$. Suppose $g=\varphi^{-1}\circ f\circ \varphi$, for some $\varphi\in z+z^2\mathbb{C}\{z\}$, and suppose that there exists $r\in\mathbb{C}\{z\}$ such that $\varphi^{-1}=\text{Id}+r\circ f-r$. Let $(g_1^{f},g_2^{f})$ and $(g_1^{g},g_2^{g})$ denote $1$-moments for $f$ and $g$ respectively. We will prove that they coincide, up to identifications \eqref{idii}.
From $1$-Abel equation for $g$, after the change of variables and then using \eqref{classi}, we get
$$
\big(H^g\circ\varphi^{-1}+r\big)\circ f-\big(H^g\circ\varphi^{-1}+r\big)=-\text{Id}.
$$
By uniqueness of the formal solution of $1$-Abel equation for $f$ up to a constant term $C\in\mathbb{C}$, we get
\begin{equation}\label{jaja}
\widehat{H}^f=\widehat H^g\circ\varphi^{-1}+r+C.
\end{equation}
Since $z\mapsto r(z)+C$ is analytic, from \eqref{jaja}, we have that (up to constant term from the choice of sectorial solutions)
\begin{equation}\label{t1}
H^f_+(z)-H^f_-(z)=(H^g_+-H^g_-)\circ\varphi^{-1}(z),\ z\in V^{up}\cup V^{low}.
\end{equation}
By Remark~\ref{idiot}, for conjugation $\varphi$ above, there exists a choice of trivialisations (athat is, of constant terms) $\Psi_+^f$ and $\Psi_+^g$, such that $\Psi_+^g=\Psi_+^f\circ\varphi$. Then, for $1$-moments with respect to these trivialisations, the following property holds:
\begin{align}\label{t2}
&H^f_+(z)-H^f_-(z)=g_1^f(e^{2\pi i \Psi_+^f(z)}),\\
&H^g_+\circ\varphi^{-1}(z)-H^g_-\circ\varphi^{-1}(z)=g_1^g(e^{2\pi i \Psi_+^f(z)}),\ z\in V^{up}.\nonumber
\end{align}
By \eqref{t1} and \eqref{t2}, and repeating the same procedure for $g_2^{f,g}$ on $V^{low}$, we get that the $1$-moments coincide (up to identifications \eqref{idii}).
$\hfill\Box$
\medskip

\medskip
We finish the section with comment about relative positions of analytic and $1$-conjugacy classes.
Proposition~\ref{chartriv} and Theorem~\ref{bijectiveness} put together, we conclude that the appropriately quotiented trivial analytic class in fact \emph{parametrizes} the set of all $1$-conjugacy classes. By Theorem~\ref{bijectiveness} and Proposition~\ref{formclas}, we see that analytic classes and $1$-conjugacy classes lie in \emph{transversal position}. Each analytic class spreads through all $1$-conjugacy classes. Each $1$-conjugacy class spreads through all analytic classes, such that in each analytic class it has infinitely many representatives. In particular, in model analytic class $\mathcal{C}_0$ there exist diffeomorphisms from all $1$-conjugacy classes, and in trivial $1$-conjugacy class $\mathcal{S}$ there exist diffeomorphisms from all analytic classes. This is consistent with Example~\ref{ex2}. All this supports our observation from Section~\ref{fourhalf} that analytic class cannot be read only from differences of sectorial solutions of $1$-Abel equation on intersections of petals.

\smallskip
We finish the section by noting that the same classification analysis could have been performed considering the moments with respect to trivializations $\Psi_-^{-1}$ on negative instead on positive sectors.

\subsection{Reconstruction of the analytic classes from the $1$-conjugacy classes with respect to trivializations of both sectors.}\label{seventhree}\

We have seen in Subsection~\ref{sevenone} that we cannot read the analytic classes from the $1$-conjugacy classes with respect to positive trivialisations (or with respect to negative trivialisations). Nevertheless, by comparing the $1$-conjugacy classes of a diffeomorphism \emph{with respect to both sectorial trivializations}, in cases where the $1$-moments are invertible, we can read the analytic class. This is nothing unexpected, since comparing the sectorial trivializations themselves reveals the analytic class.

Let us recall the standard definition of \emph{\' Ecalle-Voronin moduli of analytic classification} from \cite{ecalle}, \cite{voronin} or \cite{loray},\ \cite{dudko}. For germ $f$ inside the model formal class, its analytic class is given by pair of diffeomorphisms $(\varphi_0,\varphi_\infty)\in \text{Diff}(\overline{\mathbb{C}},0)\times \text{Diff}(\overline{\mathbb{C}},\infty)$, after identifications up to post and premultiplication of both germs by arbitrary nonzero constants, corresponding to different choices of constants in sectorial trivialisations:
\begin{align}\label{mod}
\varphi_0(t)&=e^{{-2\pi i}\Psi_-\circ (\Psi_+)^{-1}\left(-\frac{Log t}{2\pi i}\right)},\ \ t\approx 0,\\
\varphi_\infty(t)&=e^{{-2\pi i}\Psi_-\circ (\Psi_+)^{-1}\left(-\frac{Log t}{2\pi i}\right)},\ t\approx\infty.\nonumber
\end{align}

Here, $\Psi_+$ and $\Psi_-$ are sectorial trivialisation functions for $f$ (solutions of Abel equation for $f$), unique up to additive constant.

 We denote by $\big(g_\infty^+,g_0^+\big)$ its $1$-moment with respect to trivialisations of the attracting sector, and by  $\big(g_\infty^-,g_0^-\big)$ its $1$-moment with respect to trivialisations of the repelling sector. 
\begin{proposition}[Ecalle-Voronin moduli expressed using $1$-moments with respect to both trivialisations]
If all $1$-moment components $g_{\infty,0}^{\pm}$ are \emph{invertible} at zero, the \' Ecalle-Voronin moduli \eqref{mod} correspond to compositions
\begin{align*}
\varphi_0(t)&=(g_{0}^-)^{-1}\circ g_0^+(t),\ t\approx 0,\\
\varphi_\infty(t)&=(g_{\infty}^-\circ \tau)^{-1}\circ (g_\infty^+\circ\tau)(t),\ \ t\approx \infty.\nonumber
\end{align*} Here $\tau(t)=1/t$ denotes the inversion.
\end{proposition}

\begin{proof} By definition of $1$-moments, it holds that
\begin{align*}
H_+(z)-H_-(z)&=g_\infty^+(e^{2\pi i \Psi_+(z)})=g_\infty^-(e^{2\pi i \Psi_-(z)}),\ z\in V^{up},\\
H_-(z)-H_+(z)&=-2\pi i+g_0^+(e^{-2\pi i\Psi_+(z)})=-2\pi i+g_0^-(e^{-2\pi i\Psi_-(z)}),\\
&\hspace{7cm} z\in V^{low}.
\end{align*}
The statement now follows directly from definition of moduli \eqref{mod}. 
\end{proof}

We address the issue of \emph{invertibility} of $g_{\infty,0}^\pm$. It relies on nontriviality of $1$-moments in both components.
Suppose 
\begin{equation*}
H_+-H_-\equiv\!\!\!\!\!/\  0\ \text{\ on\ } V^{up},\ \ \text{ and }\quad H_+-H_-\equiv\!\!\!\!\!\!/\ 2\pi i\ \text{\ on\ } V^{low}.
\end{equation*}
It can be easily seen that the germs $g_{\infty,0}^\pm$ are either diffeomorphisms or have finitely many analytic (except at zero) inverses. One of the analytic inverses gives moduli. For example, if we choose trivialisations $\Psi_+$ and $\Psi_-$ with the same constant term, then $\varphi_0$ and $\varphi_\infty$ from \eqref{mod} are \emph{tangent to the identity}, so we know which inverse to choose.

On the other hand, if the difference $H_+-H_-$ is trivial on either $V^{up}$ or on $V^{low}$, then the moduli cannot be reconstructed using the differences $H_+-H_-$ on $V^{up}\cup V^{low}$ in the above manner, and the analytic class cannot be reconstructed.

\section{Prospects}\label{five}
\subsection{Can we recognize a diffeomorphism using directed areas of $\varepsilon$-neighborhoods of \emph{only one} orbit?}\label{fiveone}\ 

We exploit ideas from the proof of Proposition~\ref{accu} to prove that a germ $f$ is uniquely determined by function $\varepsilon\mapsto A^{\mathbb{C}}(z_0,\varepsilon)$, $\varepsilon\in(0,\varepsilon_0)$, where $\varepsilon_0>0$ is arbitrary small. Note that $z_0\in V_+$ is fixed and this function is realized using \emph{only one orbit}. This result suggests that one orbit should be enough to read the analytic class. Therefore, it should suffice to fix $z$ and regard $A^\mathbb{C}(z,\varepsilon)$ as function of $\varepsilon$ only. However, how this can be done remains open and subject to further research. Note that this is a different approach to the problem; in the article, we have been considering and comparing sectorial functions, derived from $A^\mathbb{C}(z,\varepsilon)$ with respect to variable $z$. 

Let $\text{Diff\,}(\mathbb{C},0;z_0)\subset \text{Diff\,}(\mathbb{C},0)$ denote the set of all parabolic germs whose basin of attraction contains $z_0$.
\begin{proposition}\label{undet}
Let $z_0\in V_+$ be fixed. Let $\varepsilon_0>0$. The mapping
$$
f\in \text{Diff\,}(\mathbb{C},0; z_0)\longmapsto \big(\varepsilon\mapsto A^\mathbb{C}(z_0,\varepsilon),\ \varepsilon\in(0,\varepsilon_0)\big)
$$
is injective on the set $\text{Diff\,}(\mathbb{C},0; z_0)$.  
\end{proposition}

\begin{proof}
Suppose that 
$
A^{\mathbb{C},f}(z_0,\varepsilon)=A^{\mathbb{C},g}(z_0,\varepsilon),\ \varepsilon\in (0,\varepsilon_0)$, for some $f,\ g\in\text{Diff\,}(\mathbb{C},0; z_0)$.
 We show that the germs $f$ and $g$ must be equal.

Separating the tails and the nuclei and dividing by $\varepsilon^2\pi$, we get
\begin{equation}\label{sep}
\frac{A^{\mathbb{C},f}(T_\varepsilon)-A^{\mathbb{C},g}(T_\varepsilon)}{\varepsilon^2\pi}=\frac{A^{\mathbb{C},g}(N_\varepsilon)-A^{\mathbb{C},f}(N_\varepsilon)}{\varepsilon^2\pi},\ \varepsilon\in(0,\varepsilon_0).
\end{equation}
The proof relies on the presence of singularities of directed areas at $(\varepsilon_n^{f,g})$, where $(\varepsilon_n^{f,g})$ are as defined in Proposition~\ref{accu}. Let $z_n,\ n\in\mathbb{N}$, denote the points of orbit $S^f(z_0)$ of $f$ and $w_n,\ n\in\mathbb{N}$, of orbit $S^g(z_0)$ of $g$. Recall that
$$
\varepsilon_n^f=\frac{|z_{n}-z_{n+1}|}{2},\ \varepsilon_n^g=\frac{|w_{n}-w_{n+1}|}{2},\ \ n\in\mathbb{N}.
$$

Suppose that the sequences of singularities for $f$ and $g$, $(\varepsilon_n^f)$ and $(\varepsilon_n^g)$, do not eventually coincide. Then there exists $n$ arbitrary big and an interval $(\varepsilon_n^f-\delta,\varepsilon_n^f+\delta),\ \delta>0$, such that $\varepsilon_{m}^g<\varepsilon_n^f-\delta$ and $\varepsilon_n^f+\delta<\varepsilon_{m-1}^g$. 
Consider the second derivative $\frac{d^2}{d\varepsilon^2}$ of \eqref{sep} in $\varepsilon_n^f$ from the right. With all the notations and conclusions as in the proof of Proposition~\ref{accu}, from \eqref{druga}, we have
\begin{align}\label{contra}
0=&(G_{n+1}^f)^{\prime\prime}(\varepsilon_n^f+)-{(G_{m}^g)}^{\prime\prime}(\varepsilon_n^f+)+\nonumber\\  &+\frac{1}{\pi}\bigg(\frac{4\varepsilon_n^f}{\varepsilon^3}\sqrt{1-\frac{(\varepsilon_n^f)^2}{\varepsilon^2}}-\frac{2(\varepsilon_n^f)^3}{\varepsilon^5}\frac{1}{\sqrt{1-\frac{(\varepsilon_n^f)^2}{\varepsilon^2}}}\bigg)\Bigg|_{\varepsilon=\varepsilon_n^f+}\Bigg.\cdot\big(z_{n+1}+z_n).
\end{align}
Since all terms are bounded except the term in brackets and $z_n+z_{n+1}\neq 0$, \eqref{contra} leads to a contradiction. Therefore, sequences of singularities $(\varepsilon_n^f)$ and $(\varepsilon_n^g)$ eventually coincide,
$$
\varepsilon_n^f=\varepsilon_{n+k_0}^g,\ n\geq n_0,\ k_0\in\mathbb{N}.
$$

Now, considering the second derivative \eqref{sep} at the singularity $\varepsilon_n=\varepsilon_n^f=\varepsilon_{n+k_0}^g$ from the right, instead of \eqref{contra}, we have:
\begin{align*}
&0={(G_{n+1}^f)}^{\prime\prime}(\varepsilon_n+)-{(G_{n+k_0+1}^g)}^{\prime\prime}(\varepsilon_n+)+\\  &+\frac{1}{\pi}\bigg(\frac{4\varepsilon_n}{\varepsilon^3}\sqrt{1-\frac{\varepsilon_n^2}{\varepsilon^2}}-\frac{2\varepsilon_n^3}{\varepsilon^5}\frac{1}{\sqrt{1-\frac{\varepsilon_n^2}{\varepsilon^2}}}\bigg)\Bigg|_{\varepsilon=\varepsilon_n+}\cdot\hspace{-1.2cm}\cdot \Big(z_{n+1}+z_n-(w_{n+k_0+1}+w_{n+k_0})\Big).\nonumber
\end{align*}
The term in brackets is the only unbounded term, therefore 
$
(z_{n+1}+z_n)-(w_{n+k_0+1}+w_{n+k_0})=0.
$ The middle points of the orbits $S^f(z_0)$ and $S^g(z_0)$ eventually coincide. Since the distances $d_n^f=2\varepsilon_n^f$ and $d_n^g=2\varepsilon_n^g$ coincide, and since both orbits converge to some tangential direction, it is easy to see that the orbits themselves eventually coincide.

Two diffeomorphisms $f$ and $g$, both analytic at $z=0$, coincide on a set accumulating at the origin. Therefore, they must be equal. 
\end{proof}

\subsection{Application of $1$-Abel equation in analytic classification problem of two-dimensional diffeomorphisms}\label{fivetwo}\ 

This result is due to David Sauzin (personal communication). It gives a possible application of the $1$-Abel equation for a diffeomorphism $f$, along with its Abel equation, to analytic classification of two-dimensional germs derived from $f$, of rather special form: $$F(z,w)=(f(z),z+w).$$ The classification results in dimension two are scarce. This class of diffeomorphisms is not completely artificial. The natural correspondence between cohomological equations and mappings of this form was already noted in \cite[Section 3]{beli}. In \cite[Section 4]{vorgri} similar mappings are called $w$-shifts, and their analytic classification invariants are discussed. 

Let $f$ belong to the formal class of $f_0$. We consider two-dimensional germs of diffeomorphisms $F:\mathbb{C}\times \mathbb{C}\to \mathbb{C}\times \mathbb{C}$ of the type
$$
F(z,w)=(f(z),z+w).
$$
Each two-dimensional diffeomorphism of the above type can by unique formal change of variables $\Phi(z,w)\in \mathbb{C}[[z,w]]^2$ be reduced to a formal normal form of the type
$$
F_0(z,w)=(f_0(z),z+w).
$$
Here, $\mathbb{C}[[z,w]]$ denotes a $2$-dimensional formal series, without constant term.

The formal conjugation $\widehat{\Phi}(z,w)$, $F=\widehat{\Phi}^{-1}\circ F_0\circ \widehat{\Phi}$, is given by
\begin{equation}\label{conj2}
\widehat{\Phi}(z,w)=\Big(\widehat{\varphi}(z),\ \widehat{H}(z)-\widehat{H}^{f_0}\circ\widehat{\varphi}(z)+w\Big).
\end{equation}
Here, $\widehat{\varphi}$ is the formal conjugation that conjugates $f$ to $f_0$. $\widehat{H}$ is the formal solution of $1$-Abel equation for $f$ and $\widehat{H}^{f_0}$ for $f_0$.
To conclude,
$F$ is analytically conjugated to normal form $F_0$ if and only if $f$ is analytically conjugated to $f_0$ and $$H_+-H_-\equiv \left(H_+^{f_0}-H_-^{f_0}\right)\circ \varphi,\ \text{\ on\ } V^{up}\cap V^{low},$$
the latter difference for $f_0$ being known, see Example~\ref{ex1}.

\smallskip
The problem of formal conjugacy can be formulated equivalently using trivialization equation that conjugates $F$ with translation by $(1,0)$. We search for formal solutions $\widehat{T}(z,w)$ of the trivialization equation:
\begin{equation}\label{trivic}
\widehat{T}(F(z,w))=\widehat{T}(z,w)+(1,0).
\end{equation}
It can be checked that formal solution of trivialization equation \eqref{trivic} for the normal form $F_0$ is given by
\begin{equation}\label{triv0}
\widehat{T}_0(z,w)=(\Psi^{f_0}(z),\ \widehat{H}^{f_0}(z)+w).
\end{equation}
Here, $\Psi^{f_0}(z)=-1/z$. $\widehat{H}^{f_0}$ is formal solution of $1$-Abel equation for $f_0$, sectorially analytic, explicitely given in Example~\ref{ex2} in Subsection~\ref{fivesecond}.

As in 1-dimensional case, by \eqref{conj2}, \eqref{trivic} and \eqref{triv0}, we get that formal trivialization $\widehat{T}$ for diffeomorphism $F=\widehat{\Phi}^{-1}\circ F_0\circ \widehat{\Phi}$ is given by
\begin{equation}\label{trivef}
\widehat{T}(z,w)=\widehat{T}_0\left(\widehat{\Phi}(z,w)\right)=\Big(\widehat{\Psi}(z),\widehat{H}(z)+w\Big).
\end{equation}
Here, $\widehat{\Psi}$ is the formal solution of the Abel equation for $f$.

Obviously, by \eqref{trivef}, Abel equation for $f$ appears as the first coordinate and $1$-Abel equation as the second coordinate in the trivialization equation \eqref{trivic} for $(z,w)\mapsto F(z,w)=(f(z),z+w)$.

\section{Appendix}\label{six}
 
\noindent \emph{Proof of Proposition~\ref{nonasy}}.\ We show the obstacle for the existence of a full asymptotic expansion: the index $n_\varepsilon$ separating the tail and the nucleus of the $\varepsilon$-neighborhood of the orbit does not have asymptotic expansion in $\varepsilon$ after the first $k+1$ terms. 

By \cite[Lemma 1]{resman}, $n_\varepsilon$ has the following expansion, as $\varepsilon\to 0$:
\begin{equation}\label{razne}
n_\varepsilon=p_1\varepsilon^{-1+\frac{1}{k+1}}+\ldots+p_k\varepsilon^{-1+\frac{k}{k+1}}+p_{k+1}\log\varepsilon+r(z,\varepsilon),
\end{equation}
where $r(z,\varepsilon)=O(1)$ in $\varepsilon$, for $z$ fixed. We put $z$ here only to denote the dependence of the function on the initial point. Here, $z$ is only a fixed complex number. 

Suppose that the limit $\lim_{\varepsilon\to 0}r(z,\varepsilon)$ exists. Then, 
\begin{equation}\label{ff}
r(z,\varepsilon)=C(z)+o(1),\ \varepsilon\to 0\ \ \text{\ $(C$ can be 0$)$}.
\end{equation}
In the points $\varepsilon_n$ as above, it holds
$$
n(\varepsilon_n+)=n,\ n(\varepsilon_n-)=n+1.
$$
The $(k+1)$-jet of the expansion \eqref{razne} is continuous on $(0,\varepsilon_0)$. By \eqref{ff}, $r(\varepsilon_n)=C+o(1)$, as $n\to\infty$. Therefore we get that
$$
1=n(\varepsilon_n+)-n(\varepsilon_n-)=o(1),\ n\to\infty,
$$
which is a contradiction. The limit $\lim_{\varepsilon\to 0} r(z,\varepsilon)$ does not exist.

Going through the proofs of Lemmas 4 and 5 in \cite{resman} for the expansions of the areas of the tail and the nucleus, since $n_\varepsilon$ does not have expansion after the $(k+1)$-st term, we conclude that $A^\mathbb{C}(z,\varepsilon)$ in general does not have full expansion in $\varepsilon$, as stated.
$\hfill\Box$

\medskip
Before proving Proposition~\ref{accu}, we state (without proof) an auxiliary proposition that we need in the proof. 
\begin{proposition}\label{auxi}
Let all the notations be as in Subsection~\ref{threeone}. Let $\delta>0$ such that $\varepsilon_{n+1}+\delta<\varepsilon_n$. For each $n\in\mathbb{N}$, the function $H_{n+1}$,
$$
H_{n+1}(\varepsilon)=\frac{1}{\pi}\sum_{l=n+1}^{\infty}\left[\Big(\frac{\varepsilon_{l}}{\varepsilon}\sqrt{1-\frac{\varepsilon_{l}^2}{\varepsilon^2}}+\arcsin{\frac{\varepsilon_{l}}{\varepsilon}}\Big)(z_{l}+z_{l+1})\right]
$$
is a well-defined $C^\infty$-function in $\varepsilon$ on the interval $\varepsilon\in(\varepsilon_{n+1}+\delta,\varepsilon_{n-1})$. Moreover, the differentiation of the sum is performed term by term.
\end{proposition}

\noindent \emph{Proof of Proposition~\ref{accu}}.
We analyse the directed area of the tail and of the nucleus separately. Without any change in the class in $(0,\varepsilon_0)$, we can consider the directed area divided by $\varepsilon^2\pi$. We show that the points where class $C^2$ is lost are the points $\varepsilon_n$ in which, when $\varepsilon$ decreases, one disc detaches from the nucleus to the tail. We have
\begin{equation*}
\frac{A^\mathbb{C}(z,\varepsilon)}{\varepsilon^2\pi}=\frac{A^\mathbb{C}(T_\varepsilon)}{\varepsilon^2\pi}+\frac{A^\mathbb{C}(N_\varepsilon)}{\varepsilon^2\pi}.
\end{equation*}

The function $\varepsilon\mapsto \frac{A^\mathbb{C}(T_\varepsilon)}{\varepsilon^2\pi}$ is easy to analyse: it is a piecewise constant function on the intervals $[\varepsilon_{n+1},\varepsilon_n)$, with jumps at $\varepsilon=\varepsilon_n$ of value $+z_n$.

The directed area of the nucleus is computed adding the contribution of each crescent. By Proposition 5 in \cite{resman}, we have:
\begin{align}\label{nuclis}
\frac{A^\mathbb{C}(N_\varepsilon)}{\varepsilon^2\pi}=\left\{\begin{array}{l}z_{n+1}+G_{n+1}(\varepsilon),
\hfill\varepsilon\in[\varepsilon_{n+1},\varepsilon_{n}),\\[0.3cm] z_n+G_{n+1}(\varepsilon)+\\
\quad +\frac{1}{\pi}\left(\frac{\varepsilon_{n}}{\varepsilon}\sqrt{1-\frac{\varepsilon_{n}^2}{\varepsilon^2}}+\arcsin{\frac{\varepsilon_{n}}{\varepsilon}}\right)(z_{n}+z_{n+1})+\frac{z_{n+1}-z_{n}}{2},\\\hfill \varepsilon\in[\varepsilon_n,\varepsilon_{n-1}).\end{array}\right.
\end{align} 
Here, by $G_{n+1},\ n\in\mathbb{N}$, we denote the complex functions \begin{align*}G_{n+1}(\varepsilon)=&\frac{1}{\pi}\sum_{k=n+1}^{\infty}\left(\frac{\varepsilon_{k}}{\varepsilon}\sqrt{1-\frac{\varepsilon_{k}^2}{\varepsilon^2}}+\arcsin{\frac{\varepsilon_{k}}{\varepsilon}}\right)(z_{k}+z_{k+1})+\frac{z_{k+1}-z_{k}}{2}.
\end{align*}
$G_{n+1}(\varepsilon)$ presents the sum of contributions from the crescents corresponding to the points $z_{n+2},\ z_{n+3},$ etc.

Let $\delta>0$ such that $\varepsilon_{n+1}+\delta<\varepsilon_n$.  By Proposition~\ref{auxi} in the Appendix, function $G_ {n+1}$ is of class $C^2$ on each interval $(\varepsilon_{n+1}+\delta,\varepsilon_{n-1})$, $\delta>0$. Therefore, by \eqref{nuclis}, the point of nondifferentiability of $A^\mathbb{C}(N_\varepsilon)$ on $(\varepsilon_{n+1}+\delta,\varepsilon_{n-1})$ can only be $\varepsilon=\varepsilon_n$, where two parts defined by different formulas glue together. In the sequel, we show that at point $\varepsilon=\varepsilon_n$, $A^\mathbb{C}(N_\varepsilon)$ is of class $C^1$, but not $C^2$. Differentiating \eqref{nuclis} in $\varepsilon$ on some interval around $\varepsilon_n$, we get
\begin{align*}
\frac{d}{d\varepsilon}\frac{A^\mathbb{C}(N_\varepsilon)}{\varepsilon^2\pi}\Big|_{\varepsilon=\varepsilon_n-}=G_{n+1}'(\varepsilon_n-)\Big.,\
\frac{d}{d\varepsilon}\frac{A^\mathbb{C}(N_\varepsilon)}{\varepsilon^2\pi}\Big|_{\varepsilon=\varepsilon_n+}=G_{n+1}'(\varepsilon_n+)\Big.,
\end{align*}
the two being finite and equal since $G_{n+1}$ is of the class $C^2$ around $\varepsilon_n$. Therefore, $A^\mathbb{C}(N_\varepsilon)$ is of class $C^1$ at $\varepsilon=\varepsilon_n$,\ $n\in\mathbb{N}$. 

Differentiating once again, we get
\begin{align}\label{druga}
\frac{d^2}{d\varepsilon^2}&\frac{A^\mathbb{C}(N_\varepsilon)}{\varepsilon^2\pi}\Big|_{\varepsilon=\varepsilon_n-}={(G_{n+1})}^{\prime\prime}(\varepsilon_n-),\nonumber\\
\frac{d^2}{d\varepsilon^2}&\frac{A^\mathbb{C}(N_\varepsilon)}{\varepsilon^2\pi}\Big|_{\varepsilon=\varepsilon_n+}={(G_{n+1})}^{\prime\prime}(\varepsilon_n+)\ +\nonumber\\
& \qquad +\frac{1}{\pi}\bigg(\frac{4\varepsilon_n}{\varepsilon^3}\sqrt{1-\frac{\varepsilon_n^2}{\varepsilon^2}}-\frac{2\varepsilon_n^3}{\varepsilon^5}\frac{1}{\sqrt{1-\frac{\varepsilon_n^2}{\varepsilon^2}}}\bigg)\Bigg|_{\varepsilon=\varepsilon_n+}\cdot\big(z_{n+1}+z_n).
\end{align}
Although ${(G_{n+1})}^{\prime\prime}(\varepsilon_n-)={(G_{n+1})}^{\prime\prime}(\varepsilon_n+)\in\mathbb{C}$, the other term is unbounded when $\varepsilon\to\varepsilon_n+$. Therefore, the second derivative of $A^\mathbb{C}(N_\varepsilon)$ at $\varepsilon=\varepsilon_n$, $n\in\mathbb{N}$, does not exist.

Finally, glueing overlapping intervals $(\varepsilon_{n-1}+\delta,\varepsilon_{n+1}),\ n\in\mathbb{N},$ and adding the tail and the nucleus, we get the desired result.
$\hfill\Box$

\medskip
\noindent \emph{Proof of Proposition~\ref{no}}.
Let $\varepsilon>0$ be fixed. By $U_\varepsilon$ we denote the open set $U_\varepsilon=\{z\in V_+:\ |z-f(z)|<2\varepsilon\}$. For $z\in U_\varepsilon$, the $\varepsilon$-discs centered at points $z$ and $f(z)$ in $S^f(z)_\varepsilon$ overlap. Therefore, the $\varepsilon$-neighborhoods of orbits $S^f(z)$ and $S^f(f(z))$ differ by a crescent. By Proposition 5 in \cite{resman}, we get
\begin{align}\label{analy}
A^\mathbb{C}&(z,\varepsilon)=A^\mathbb{C}(f(z),\varepsilon)-\frac{\pi}{2}\varepsilon^2(f(z)-z)+\\
&\hspace{3cm}+\varepsilon^2(z+f(z))\cdot G\Big(\frac{|z-f(z)|}{2\varepsilon}\Big),\ z\in U_\varepsilon.\nonumber
\end{align}
Here, $G(t)=t\sqrt{1-t^2}+\arcsin t$, $t\in(0,1)$. We define function $T$:
\begin{equation}\label{te}
T(z)=A^\mathbb{C}(z,\varepsilon)-A^\mathbb{C}(f(z),\varepsilon),\ z\in V_+.
\end{equation}
By \eqref{analy}, it holds
$$
T(z)=-\frac{\pi}{2}\varepsilon^2(f(z)-z)+\varepsilon^2(z+f(z))\cdot G\Big(\frac{|z-f(z)|}{2\varepsilon}\Big),\ z\in U_\varepsilon.
$$
It holds that there exists some punctured neighborhood of $0$ such that $f'(z)\neq 1$, for all $z$ in that neighborhood. Otherwise, by analyticity of $f$ at $z=0$, it would hold that $f'\equiv 1$ on some neighborhood of 0. By inverse function theorem applied locally to $G$ and $\text{Id}-f$, and since absolute value is nowhere analytic, we see that $T$ is nowhere analytic on $U_\varepsilon$. 

We now take any small sector $S^+(\varphi,r)\subset V_+,\ r>0,\ \varphi\in(0,\pi)$. Suppose that $z\mapsto A^\mathbb{C}(z,\varepsilon)$ is analytic on $S_+$. Since $f$ is analytic, and $f(z)\in S_+$ for $z\in S_+$, the function $z\mapsto T(z)$ defined by \eqref{te} is also analytic on $S_+$. 
The intersection $S_+\cap U_\varepsilon$ is nonempty and therefore we derive a contradiction.
$\hfill\Box$

\medskip
\textbf{Acknowledgments}\ 
\emph{I would like to thank my supervisor, Pavao Marde\v si\' c, for proposing the subject and for numerous discussions and advices. Many thanks to David Sauzin for useful discussion. }

\end{document}